\documentclass[11pt,reqno,oneside]{amsart}
\usepackage{geometry}                
\usepackage[parfill]{parskip}    
\usepackage{graphicx}
\usepackage{amssymb}
\usepackage{epstopdf}

\usepackage{geometry}
\usepackage{hyperref}
\usepackage{graphicx}

\usepackage{booktabs} 
\usepackage{array} 
\usepackage{paralist} 
\usepackage{verbatim} 
\usepackage{subfig} 
\usepackage{tabularx}
\usepackage{fancyhdr} 
\usepackage{amsmath,amsfonts,amsthm,mathrsfs,amssymb,cite}
\usepackage[usenames]{color}

\newcommand{\R}{{\mathbb R}}

\newcommand{\C}{{\mathbb C}}

\newcommand{\no}{\nonumber}
\newcommand{\be}{\begin{eqnarray}}
\newcommand{\ben}{\begin{eqnarray*}}
\newcommand{\en}{\end{eqnarray}}
\newcommand{\enn}{\end{eqnarray*}}

\newcommand{\curl}{{\rm curl\,}}

\newcommand{\Grad}{{\rm Grad\,}}
\newcommand{\grad}{{\rm grad\,}}
\newcommand{\divv}{{\rm div\,}}

\newcommand{\ima}{{\rm Im\,}}

\newcommand{\Om}{\Omega}

\newcommand{\la}{\lambda}

\newtheorem{thm}{Theorem}[section]

\newtheorem{lem}{Lemma}[section]

\theoremstyle{definition}
\newtheorem{defn}{Definition}[section]
\theoremstyle{remark}

\newtheorem{rem}{Remark}[section]
\numberwithin{equation}{section}

\title{\bf Nearly Cloaking the Elastic Wave Fields}

\author{Guanghui Hu}
\address{Weierstrass Institute, Mohrenstr. 39,
10117 Berlin, Germany.}
\email{hu@wias-berlin.de}

\author{Hongyu Liu}
\address{Department of Mathematics,
Hong Kong Baptist University,
Kowloon, Hong Kong}
\email{hongyu.liuip@gmail.com}

\begin{document}
\maketitle

\begin{abstract}
In this work, we develop a general mathematical framework on regularized approximate cloaking of elastic waves governed by the Lam\'e system via the approach of transformation elastodynamics. Our study is rather comprehensive. We first provide a rigorous justification of the transformation elastodynamics. Based on the blow-up-a-point construction, elastic material tensors for a perfect cloak are derived and shown to possess singularities. In order to avoid the singular structure, we propose to regularize the blow-up-a-point construction to be the blow-up-a-small-region construction. However, it is shown that without incorporating a suitable lossy layer, the regularized construction would fail due to resonant inclusions. In order to defeat the failure of the lossless construction, a properly designed lossy layer is introduced into the regularized cloaking construction . We derive sharp asymptotic estimates in assessing the cloaking performance. The proposed cloaking scheme is capable of nearly cloaking an arbitrary content with a high accuracy.

\medskip

\noindent{\bf Keywords}. Elastic cloaking, transformation elastodynamics, Lam\'e system, regularization, asymptotic estimates.

\noindent{\bf Mathematics Subject Classification (2010)}:  74B05, 35R30, 35J25, 74J20.

\end{abstract}

\section{Introduction}

This paper concerns the cloaking of elastic waves. An elastic region is said to be cloaked if its content together with the cloak is ``unseen" by the exterior elastic wave detections. In recent years, the study on elastic cloaking has gained growing interests in the physics literature (cf. \cite{BGM,DG, DG03, FGEM,MBW, Norris11, Parnell}), much followed the development of transformation-optics cloaking of optical waves including the acoustic and electromagnetic waves. A proposal for cloaking for electrostatics using the invariance properties
of the conductivity equation was pioneered in \cite{GLU,GLU2}.
Blueprints for making objects invisible to electromagnetic (EM) waves were proposed in two articles in {\it Science} in 2006 \cite{Leo,PenSchSmi}. The article by Pendry et al
uses the same transformation used in \cite{GLU,GLU2} while the work of Leonhardt uses a conformal mapping in two dimensions. The method based on the invariance properties of the equations modeling the optical wave phenomenon has been named
{\sl transformation optics} and has received a lot of attention in the scientific community due to its significant practical importance. We refer to the survey articles \cite{CCr,GKLU4,GKLU5,LZ} and the references therein for the theoretical and experimental progress on optical cloaking.

The Lam\'e system governing the elastic wave propagation also possesses a certain transformation property, in a more complicated manner than that for the optical wave equations.
Using the transformation property, the transformation-elastodynamics approach can be developed for the construction of elastic cloaks, following a similar spirit to the transformation-optics construction of optical cloaks. In a rather heuristic way, an ideal invisibility cloak can be obtained by the blow-up-a-point construction proceeded as follows. One first selects a region $\Omega$ in the homogeneous space for constructing the cloaking device. Let $P\in\Omega$ be a single point and let $F$ be a diffeomorphism which blows up $P$ to a region $D$ within $\Omega$. Using transformation-elastodynamics, the ambient homogeneous medium around $P$ is then `compressed' via the push-forward to form the cloaking medium in $\Omega\backslash\overline{D}$, whereas the `hole' $D$ forms the cloaked region within which one can place the target object. The cloaking region $\Omega\backslash\overline{D}$ and the cloaked region $D$ yield the cloaking device in the physical space, whereas the homogeneous background space containing the singular point $P$ is referred to as the virtual space. Due to the transformation invariance of the elastic system, the exterior measurements corresponding to the cloaking device in the physical space are the same to those in the virtual space corresponding to a singular point. Intuitively speaking, the scattering information of the elastic cloak is `hidden' in a singular point $P$.

However, the blow-up-a-point construction would yield server singularities for the cloaking elastic material tensors. Most of the physics literature accepts the singular structure and focuses more on the application side (cf. \cite{BGM,DG,DG03,Norris11}). To our best knowledge, there is very little mathematical study on rigorously dealing with the singular elastic cloaking problem. On the other hand, there are a few mathematical works seriously dealing with the singular cloaking problems associated with the optical cloaks. Concurrently, there are two theoretical approaches in the literature: one approach is to accept the singularity and propose to investigate the physically meaningful solutions, i.e. finite energy solutions, to the singular acoustic and electromagnetic wave equations (see \cite{GKLU,LZ}); the other approach is to regularize the singular ideal cloaking construction and investigate the near-invisibility instead; see \cite{Ammari1,KSVW} on the treatment of electrostatics, \cite{Ammari3,Ammari2,KOVW,Liu,Liu1,LS} on acoustics, and \cite{Ammari4,BL,BLZ} on electromagnetism. In this work, we follow the latter approach to develop a general framework of constructing near-cloaks for elastic waves via the transformation-elastodynamics approach.

The present study on regularized approximate cloaking of elastic waves is rather comprehensive and includes several salient ingredients. First, we provide a rigorous justification of the transformation elastodynamics, which lacks in the physics literature. Particularly, we prove the well-posedness of the transformed Lam\'e system. This is presented in Section 2. In Section 3, we consider the elastic cloaking problem, and based the blow-up-a-point transformation, we give the construction of an ideal elastic cloak and analyze the singularity of the cloaking elastic material parameters. In Section 4, we introduce the regularized construction based on a blow-up-a-small-region transformation. Then, we show the existence of resonant inclusions which defy any attempt in achieving near-cloaks. Section 5 is devoted to the development of our near-cloaking scheme by incorporating a properly designed lossy layer into the regularized construction. We derive sharp estimate in assessing the cloaking performance. The asymptotic estimate is independent of the elastic content in the cloaked region, which means the proposed cloaking scheme is capable of nearly cloaking an arbitrary elastic object. The estimate is based on the use of a variety of variational arguments and layer-potential techniques. We also verify that the proposed lossy layer is a finite realization of the a \emph{traction-free} lining.

{Finally, we would like to mention in passing that our study may find important applications in seismic metamaterials (cf. \cite{A1,A2,A3,A4}) to construct feasible devices for protecting key structures from the catastrophic destruction of natural earthquake waves or terrorist attacks (e.g., nuclear blast). For instance,  the elastic invisibility cloak could be of great significance in safeguarding nuclear power plants, electric pylons, oil refineries, nuclear reactors and old or fragile monuments as well as the important components within them.}

\section{Lam\'e system and transformation elastodynamics}

Consider the time-harmonic elastic wave propagating through an anisotropic medium
occupying a bounded smooth domain $\Omega\subset \R^N$ ($N=2,3$). In linear elasticity, the spatially-dependent displacement vector $u(x)=(u_1,\cdots, u_N)(x)$ is governed by
 the following boundary value problem of the reduced Lam\'e system
\be\label{I}\left\{\begin{array}{lll}
\displaystyle\sum_{j,k,l=1}^N \frac{\partial}{\partial x_j} \left( C_{ijkl}(x)  \frac{\partial u_k}{\partial x_l}  \right)    +\omega^2\rho(x)\, u_i=0,&& \mbox{in}\quad \Omega,\quad i=1,2, \cdots, N, \\
\mathcal{N}_{\mathcal{C}}u=\psi\in H^{-1/2}(\partial \Omega)^N,&& \mbox{on}\quad \partial \Omega,\\
\end{array}\right.
\en
where $\omega$ denotes the frequency and the Neumann data $\mathcal{N}_{\mathcal{C}}u$ is defined as
\ben
\mathcal{N}_{\mathcal{C}}u:=\left( \sum_{j,k,l=1}^N \nu_j C_{1jkl}  \frac{\partial u_k}{\partial x_l},\;  \sum_{j,k,l=1}^N \nu_j C_{2jkl}  \frac{\partial u_k}{\partial x_l},\, \cdots,\,
\sum_{j,k,l=1}^N \nu_j C_{Njkl}  \frac{\partial u_k}{\partial x_l}
  \right),
\enn
with $\nu=(\nu_1,\nu_2,\ldots,\nu_N)\in\mathbb{S}^{N-1}$ denoting the exterior unit normal vector to $\partial\Omega$. In (\ref{I}), $\mathcal{C}=(C_{ijkl})_{i,j,k,l=1}^N$  is a fourth-rank constitutive material tensor of the elastic medium which shall be referred to as the stiffness tensor. $\rho$ is a complex-valued function with $\Re \rho>0$ and $\Im \rho\geq 0$, respectively, denoting the density and damping parameter of the elastic medium.
 In this paper, we employ the notation $\{\Omega; \mathcal{C}, \rho\}$ to denote the elastic medium supported in $\Omega$ characterized by the
stiffness tensor $\mathcal{C}$ with entries $C_{ijkl}(x)\in L^\infty(\Omega)$ and $\rho\in L^\infty(\Omega)$. The stiffness tensor satisfies the following symmetries for a generic anisotropic elastic material:
\begin{equation}\label{eq:symmetry}
\mbox{major symmetry:}\quad C_{ijkl}=C_{klij},\qquad {\color{black}\mbox{minor symmetries:}\quad C_{ijkl}=C_{jikl}=C_{ijlk},}
\end{equation}
for all $i,j,k,l=1,2, \cdots,N$. By Hooke's law,  the stress tensor $\sigma$ relates with the stiffness tensor $\mathcal{C}$ via the identity $\sigma(u):=\mathcal{C}:\nabla u$,
where the action of $\mathcal{C}$ on a matrix $A=(a_{ij})$ is defined as
\ben
\mathcal{C}:A=(\mathcal{C}:A)_{ij}=\displaystyle\sum_{k,l=1}^N C_{ijkl}\; a_{kl}.
\enn
Hence, the elliptic system in (\ref{I}) can be restated as
\ben
\nabla\cdot (\mathcal{C}: \nabla u)+\omega^2\rho u=0\quad\mbox{in}\quad \Omega.
\enn Moreover, the boundary operator in (\ref{I}) can be rewritten as $\mathcal{N}_{\mathcal{C}}u=\nu\cdot \sigma(u)=\nu\cdot (\mathcal{C}:\nabla u)$, which is exactly the \emph{stress vector} or \emph{traction} on $\partial \Omega$.

The equivalent variational formulation of (\ref{I}) reads as follows: find $u=(u_1,\cdots,u_N)\in H^1(\Omega)^N$ such that
\be\label{variational}
a_{\mathcal{C}}(u,v):=\int_{\Omega}\left\{  \displaystyle\sum_{i,j,k,l=1}^N  C_{ijkl}  \frac{\partial u_k}{\partial x_l} \frac{\partial \overline{v_i}}{\partial x_j}
 -\omega^2 \rho(x) u_i\overline{v}_i\right\} dx= \int_{\partial \Omega}\psi\cdot\overline{v} dx,
\en
for any $v=(v_1,v_2,\cdots,v_N)\in H^1(\Omega)^N$. Suppose further that the elastic tensor $\mathcal{C}$ satisfies the
uniform Legendre ellipticity condition {
\be\label{LH}
\sum_{i,j,k,l=1}^N C_{ijkl}(x)\, a_{ij} a_{kl} \geq c_0 \sum_{i,j=1}^N |a_{i,j}|^2,\quad a_{ij}=a_{ji},
\en for all $x\in \Omega$,
i.e., $(\mathcal{C}(x):A):A\geq c_0 ||A||^2$ for all symmetry matrices $A=(a_{ij})^N_{i,j=1}\in \R^{N\times  N}$.}
 Then the sesquilinear form on the left hand side of (\ref{variational}) satisfies G{\aa}rding's inequality
\ben
 a_{\mathcal{C}}(u,u)\geq c_0 ||\nabla u||_{L^2(\Omega)^N}-\omega^2\,||\rho||_{L^\infty(\Omega)} || u||_{L^2(\Omega)^N} \quad\mbox{for all}\quad u\in H^1(\Omega)^N.
\enn
 As a consequence, there exists a unique weak solution to (\ref{variational}) for all frequencies $\omega\in \R_+$ excluding possibly a discrete set $\mathcal{D}$ with the only accumulating point at infinity. The well-posedness of the boundary value problem (\ref{I})
allows one to define the boundary Neumann-to-Dirichlet (NtD) map as follows
$$\Lambda_{\mathcal{C},q}: H^{-1/2}(\partial \Omega)^N\rightarrow H^{1/2}(\partial \Omega)^N, \quad \Lambda_{\mathcal{C},q} f=u|_{\partial \Omega},$$
where $u\in H^{1/2}(\partial \Omega)^N$ is the unique solution to (\ref{I}). Throughout the rest of the paper, we refer to an elastic medium $\{\Omega;\mathcal{C},\rho\}$ as regular if it satisfies the major symmetry in \eqref{eq:symmetry} and the uniform Legendre ellipticity condition in \eqref{LH}, otherwise it is called singular. We note that for a regular elastic medium, the corresponding Lam\'e system is well-posed provided $\omega\notin\mathcal{D}$.

If an elastic material is isotropic and homogeneous, one has
\be\label{C}
\mathcal{C}(x)\equiv\mathcal{C}^{(0)},\quad C^{(0)}_{ijkl}=\lambda \delta_{i,j}\delta_{k,l}+\mu ( \delta_{i,k}\delta_{j,l}+ \delta_{i,l}\delta_{j,k}).
\en
That is, the stiffness tensor is constant throughout the material with the Lam\'e constants $\lambda$ and $\mu$ satisfying $\mu>0, N \lambda+2 \mu>0$. For simplicity, the mass density is usually normalized to be one in an isotropic homogeneous medium,  i.e., $\rho(x)\equiv 1$.
Under these assumptions,
 the stress tensor takes the form
\ben
\sigma(u)=\lambda\, \textbf{I}\, \divv u + 2\mu \epsilon(u),\quad \epsilon(u):=\frac{1}{2}\left(\nabla u+\nabla u^\top\right),
\enn where $\textbf{I}$ stands for the $N\times N$ identity matrix. In this case,
the Lam\'e system (\ref{I}) reduces to the boundary value problem for Navier's equation,
\be\label{eq:Naiver}
 \mathcal{L}u+ \omega^2   u = 0 \quad \mbox{in}\quad \Omega,\quad
Tu =\psi \quad \mbox{on}\quad \partial \Omega,
\en
where $Tu=T_{\lambda,\mu}u:= \nu\cdot (\mathcal{C}^{(0)}:\nabla u)$ stands for the traction on the boundary of the isotropic medium $\{\Omega; \mathcal{C}^{(0)}, 1\}$, and
\ben
\mathcal{L}u:=\nabla\cdot (\mathcal{C}^{(0)}: \nabla u)=\mu\Delta u+ (\lambda + \mu) \, \grad \divv\, u.
\enn
In two dimensions, $Tu$ can be explicitly written as
\be\label{stress-2D}
Tu:=
2 \mu \, \partial_{\nu} u + \lambda \,
\nu \, \divv u
+\mu \tau(\partial_2 u_1 - \partial_1 u_2),\quad  \tau:=(-\nu_2,\nu_1),\,\nu=(\nu_1,\nu_2),
\en
whereas in three dimensions,
\be\label{stress-3D}
Tu:=2 \mu \, \partial_{\nu} u + \lambda \,
\nu \, \divv u+\mu \nu\times \curl u,\quad \nu=(\nu_1,\nu_2,\nu_3).
\en Here and in what follows, we write $T_{\lambda,\mu}u=Tu$ to drop the dependance of $T_{\lambda,\mu}$ on the Lam\'e constants $\lambda$ and $\mu$.  Moreover, we shall refer to $\{\Omega; \mathcal{C}^{(0)}, 1\}$ as the free space or reference space in our subsequent study on the invisibility cloaking.

Let $\tilde{x}=F(x): \Omega\rightarrow \tilde{\Omega}$ be a bi-Lipschitz and orientation-preserving mapping. Our cloaking study shall be based on the following transformation invariance of the Lam\'e system (\ref{I}).

\begin{lem}\label{Lem:change}
(i) The function $u\in H^1(\Omega)^N$ is a solution to $\nabla\cdot (\mathcal{C}: \nabla u)+\omega^2\rho u=0$ in $\Omega$ if and only if $\tilde{u}=(F^{-1})^*u:=u\circ F^{-1}\in H^1(\tilde\Omega)^N$ is a solution to
\be\label{transform-Lame}
\tilde{\nabla}\cdot (\tilde{\mathcal{C}}: \tilde{\nabla} \tilde{u})+\omega^2\tilde{\rho} \tilde{u}=0\quad\mbox{in}\quad\tilde{\Omega},
\en where $\tilde{\nabla}={\nabla}_{\tilde{x}}$ and the transformed tensor and density are given by
\begin{equation}\label{eq:pushforward}
\begin{split}
\tilde{\mathcal{C}}=\tilde{C}_{iqkp}(\tilde{x})=&\frac{1}{\det(M)}\left\{\displaystyle\sum_{l,j=1}^N  C_{ijkl}  \frac{\partial \tilde{x}_p}{\partial x_l}
\frac{\partial \tilde{x}_q}{\partial x_j}\right\}\bigg |_{x=F^{-1}(\tilde{x})}=: F_{*}\mathcal{C},\\
 \tilde{\rho}=&\left(\frac{\rho}{\det(M)}\right)\bigg |_{x=F^{-1}(\tilde{x})}=: F_{*}\rho,\quad M=\left(\frac{\partial \tilde{x}_i}{\partial x_j}\right)_{i,j=1}^N.
\end{split}
\end{equation}
(ii) If the boundary $\partial \Omega$ remains fixed under the transformation, i.e., $F=\mbox{Identity}$ on $\partial \Omega$, then $\Lambda_{\mathcal{C},\rho}=\Lambda_{\tilde{\mathcal{C}},\tilde{\rho}}$.
\end{lem}
\begin{proof}
By changing the variables $\tilde{x}=F(x)$ in the sesquilinear form of (\ref{variational}) and using Green's formula, one has
\be\no
\int_{\partial \Omega}\psi\cdot\overline{v} ds &=&
\int_{\tilde{\Omega}}\left\{  \displaystyle\sum_{i,q,k,p=1}^N  \tilde{\mathcal{C}}_{iqkp}  \frac{\partial \tilde{u}_k}{\partial \tilde{x}_p} \overline{\frac{\partial \tilde{v}_i}{\partial \tilde{x}_q}}
 -\omega^2 \tilde{\rho}(\tilde{x}) \tilde{u}_i\overline{\tilde{v}}_i\right\}\, d\tilde{x} \\ \label{eq:1}
 &=&
-\int_{\tilde{\Omega}}\left\{  \divv (\tilde{\mathcal{C}}: \nabla \tilde{u}) + \omega^2\tilde{\rho}\tilde{u}\right\} \cdot \tilde{v}\,d\tilde{x}
+\int_{\partial \tilde{\Omega}}\mathcal{N}_{\mathcal{C}}\tilde{u}\cdot\overline{\tilde{v}}\, ds.
\en
By choosing $v\in C_0^\infty (\Omega)^N$, we see $\tilde{v}\in H_0^1 (\tilde{\Omega})^N$ and hence \eqref{eq:1} readily implies that  $\divv (\tilde{\mathcal{C}}: \nabla \tilde{u}) + \omega^2\tilde{\rho}\tilde{u}=0$ in $\tilde{\Omega}$.

Before proceeding to prove the second assertion, we verify the uniform Legendre elliptic condition for the
transformed tensor $\tilde{\mathcal{C}}_{iqkp}$. Indeed, for
any symmetric matrix $A=\{a_{i,j}\}$ it holds that
\be\no
 \displaystyle\sum_{i,q,k,p=1}^N  \tilde{\mathcal{C}}_{iqkp}\, a_{iq}\,a_{kp} &=&\frac{1}{\mbox{det}(M)}\displaystyle\sum_{i,q,k,p,l,j=1}^N C_{ijkl}\,\frac{\partial \tilde{x}_p}{\partial x_l}
\frac{\partial \tilde{x}_q}{\partial x_j}\, \, a_{iq}\,a_{kp}\\ \label{eq:2}
&=&\frac{1}{\mbox{det}(M)}\displaystyle\sum_{i,j,l,k=1}^N C_{ijkl}\,\, \tilde{a}_{ij}\,\tilde{a}_{kl}
\en
with
\ben
\tilde{a}_{ij}=\displaystyle\sum_{q=1}^N\frac{\partial \tilde{x}_q}{\partial x_j}\,a_{iq},\quad i, j=1,2,\cdots, N.
\enn
In view of the Legendre elliptic condition for $\mathcal{C}$ and the bi-Lipschitz assumption on $F$, we deduce from (\ref{eq:2}) that
\ben
 \displaystyle\sum_{i,q,k,p=1}^N  \tilde{C}_{iqkp}\, \tilde{a}_{ij}\,\tilde{a}_{kl}\geq c_0\,\sum_{i,j=1}^N |\tilde{a}_{ij}|^2\,\geq \tilde{c}_0\,\sum_{i,j=1}^N |a_{ij}|^2
\enn
for some constant $\tilde{c}_0>0$. That is, the transformed tensor $\tilde{\mathcal{C}}$ satisfies the Legendre elliptic condition. Therefore, the transformed Lam\'e system is well-posed and particularly we have a well-defined NtD map  $\Lambda_{\tilde{\mathcal{C}}, \tilde{\rho}}: H^{-1/2}(\tilde{\Omega})^N\rightarrow H^{1/2}(\tilde{\Omega})^N $ associated with the transformed system.

Finally, suppose that $F=\mbox{Identity}$ on $\partial \Omega$ and $\mathcal{N}_{\tilde{\mathcal{C}}} \tilde{u}=\mathcal{N}_{\mathcal{C}} u=\psi$ on $\partial \Omega$ for some $\psi\in H^{-1/2}(\tilde{\Omega})^N$. Then, one has
\ben
\Lambda_{\tilde{\mathcal{C}}, \tilde{\rho}}\psi=\tilde{u}|_{\partial \Omega}=(u\circ F^{-1})|_{\partial \Omega}=u|_{\partial \Omega}=\Lambda_{\mathcal{C}, \rho}\psi,
\enn
which readily implies that $\Lambda_{\tilde{\mathcal{C}}, \tilde{\rho}}=\Lambda_{\mathcal{C}, \rho}$.

The proof is complete.
\end{proof}

\begin{rem}\label{remark-2.1}
The transformed elastic tensor $\tilde{\mathcal{C}}$ possesses only the major symmetry, i.e.,
\ben
\tilde{C}_{iqkp}&=&\frac{1}{\det(M)}\displaystyle\sum_{l,j=1}^N  C_{ijkl}  \frac{\partial \tilde{x}_p}{\partial x_l}
\frac{\partial \tilde{x}_q}{\partial x_j}
=\frac{1}{\det(M)}\displaystyle\sum_{l,j=1}^N  C_{klij}  \frac{\partial \tilde{x}_p}{\partial x_l}
\frac{\partial \tilde{x}_q}{\partial x_j}\\
&=& \frac{1}{\det(M)}\displaystyle\sum_{j,l=1}^N  C_{kjil}  \frac{\partial \tilde{x}_p}{\partial x_j}
\frac{\partial \tilde{x}_q}{\partial x_l} =\tilde{C}_{kpiq},
\enn
where the second equality follows from the major symmetry of $\mathcal{C}$. However, $\tilde{\mathcal{C}}$ does not possess the minor symmetry. In fact, it has been pointed out by Milton et. al. \cite{MBW} that the invariance of the Lam\'e system can be achieved only if one relaxes the assumption on the minor symmetries of the transformed elastic tensor. This has led Norris \& Shuvalov \cite{Norris11} and Parnell \cite{Parnell} to explore the elastic cloaking by using Cosserat material or by employing non-linear pre-stress in a neo-Hooken elastomeric material. Design of transformation-elastodynamics-based Cosserat elastic cloaks (without the minor symmetry) has been numerically tested in the
cylindrical case \cite{BGM} as well as in the spherical case \cite{DG}.
Note that the transformed equation (\ref{transform-Lame})
retains its original form of the Lam\'e system and avoids any coupling between stress and velocity. Furthermore,
 the transformed mass density is still isotropic. We refer to \cite{MBW, Norris11, DG, BGM} for discussions and investigations of the form of  the elastodynamic equations under general transformations.
\end{rem}

\begin{rem}
In \eqref{eq:pushforward}, $F_*\mathcal{C}$ and $F_*\rho$ are called the push-forwards of $\mathcal{C}$ and $\rho$, respectively. For notational convenience, we shall write $\{\tilde\Omega; \tilde{\mathcal{C}},\tilde\rho\}=F_*\{\Omega;\mathcal{C},\rho\}$ to denote the push-forward defined in \eqref{eq:pushforward}.
\end{rem}

\section{Elastic cloaking and blowup construction}\label{sect:3}

We are in a position to introduce the elastic cloaking for our study. Henceforth, we let $\Omega\subset\mathbb{R}^N$ and $D\Subset\Omega$ be bounded and connected smooth domains. It is further assumed that $\Omega\backslash\overline{D}$ is connected and $D$ contains the origin. Let $h\in\mathbb{R}_+$ be sufficiently small and $D_h:=\{hx; x\in D\}$. Let $D_{1/2}$ represent the region which we intend to cloak and let
\ben
\{D_{1/2}; \mathcal{C}^{(a)}, \rho^{(a)}\}
\enn be the target medium. From a practical viewpoint, throughout the present study, we assume that $\{D_{1/2}; \mathcal{C}^{(a)}, \rho^{(a)}\}$ is arbitrary but regular. Let
\begin{equation}\label{device-0}
\{\Omega\backslash\overline{D}_{1/2}; \mathcal{C}^{(c)}, \rho^{(c)}\}
\end{equation}
be a properly designed layer of elastic medium, which is referred to as the cloaking medium. Let
\be\label{Omega}
\{\Omega; \mathcal{C}, \rho\}=\left\{\begin{array}{lll}
 \{\Omega\backslash \overline{D}_{1/2}; \mathcal{C}^{(c)}, \rho^{(c)}\}&&\mbox{in}\quad \Omega\backslash\overline{D}_{1/2},\\
 \{D_{1/2}; \mathcal{C}^{(a)}, \rho^{(a)}\}  &&\mbox{in}\quad D_{1/2},
\end{array}\right.
\en
be the extended medium occupying $\Omega$ and let $\Lambda_{\mathcal{C},\rho}$ be the associated NtD map.
Next, we introduce the ``free" NtD map as follows. Let $v$ be the solution to the Navier equation in the free space $\{\Omega;\mathcal{C}^{(0)}, 1\}$ (cf. (\ref{eq:Naiver}))
\be\label{Lame-0}
\mathcal{L}\,v+\omega^2   v=0\quad\mbox{in}\quad \Omega,\quad Tv=\psi\in H^{-1/2}(\partial\Omega)^N\quad\mbox{on}\quad\partial \Omega.
\en
It is assumed that $-\omega^2   $ is not an eigenvalue of the elliptic operator $\mathcal{L}$ with the traction-free boundary condition on $\partial\Omega$, and hence one has a well-defined "free" NtD map
\ben
\Lambda_0 \psi=v|_{\partial \Omega}
\enn
where $v\in H^1(\Omega)^N$ solves (\ref{Lame-0}).
The solution $v$ to (\ref{Lame-0}) can be decomposed into its compressional and shear parts as
$v=v_p+ v_s$, where in three dimensions
\be\label{kp}
v_p &:=& -\frac{1}{k_p^2} \,\grad \divv v \, , \quad k_p = \omega /\sqrt{(2\mu+\lambda)}\, ,\\ \label{ks}
v_s &:=& \frac{1}{k_s^2} \, \curl\curl v \, , \quad  k_s = \omega /\sqrt{\mu}.
\en and $k_p$, $k_s$ are known as the compressional and shear wave numbers, respectively.
It is straightforward to verify that the functions $v_p$ and $v_s$ satisfy the vectorial Helmholtz equations
\be\label{helmholtz-p}
(\Delta+k_p^2) \, v_p = 0,  \quad \curl v_p=0 &&\mbox{in}\quad \Omega,\\ \label{helmholtz-s}
(\Delta+k_s^2) \, v_s=0, \quad   \divv v_s=0 && \mbox{in} \quad \Om \, .
\en
This implies that the compressional and shear waves propagate at different speeds. By the elliptic equations (\ref{helmholtz-p}) and (\ref{helmholtz-s}), one can defined another two boundary NtD maps
\ben
\Lambda_0^{(p)}:\quad  H^{-1/2}(\partial \Omega)^N\rightarrow H^{1/2}(\partial \Omega)^N,\quad
\Lambda_0^{(p)} \psi&=&v_p|_{\partial \Omega},\\
\Lambda_0^{(s)}:\quad  H^{-1/2}(\partial \Omega)^N\rightarrow H^{1/2}(\partial \Omega)^N,\quad
\Lambda_0^{(s)} \psi&=&v_s|_{\partial \Omega},
\enn
where $v_p, v_s\in H^1(\Omega)^N$ are solutions to (\ref{helmholtz-p}) and (\ref{helmholtz-s}), respectively, prescribed with the boundary values
\be\label{T}
T_p v_p= \psi, \quad T_s v_s=\psi\quad\mbox{on}\quad\partial \Omega,
\en
with the operators $T_p$ and $T_s$ given by (in three dimensions)
\ben
T_p v_p:=2 \mu \, \partial_{\nu} v_p + \lambda\,\nu\divv v_p, \,\qquad T_s v_s:=2 \mu \, \partial_{\nu} v_s +\mu \nu\times \curl v_s.
\enn
One observes that (\ref{helmholtz-s}) is equivalent to the Maxwell system $\curl \curl v^s-k_s^2 v^s=0$ in $\Omega$, hence the boundary data $\nu\times \curl v^s:=\tilde{\psi} \in H^{-1/2}(\mbox{Div},\partial\Omega)$  is sufficient to uniquely determine $v_s\in H(\curl, \Omega)$; we refer to \cite{BCS} for the definition of Sobolev spaces mentioned here.
Since boundary value problems (\ref{helmholtz-p}), (\ref{helmholtz-s}) and
(\ref{T}) are not always solvable for general $\psi\in H^{-1/2}(\partial \Omega)^N$, we define the admissible sets of inputs by
\ben
\mathcal{P}&:=&\{ \psi\in H^{-1/2}(\partial \Omega)^N:
\mbox{there exists a $v_p$ to (\ref{helmholtz-p}) such that $T_pv_p=\psi$}
 \},\\
\mathcal{S}&:=&\{ \psi\in H^{-1/2}(\partial \Omega)^N:
\mbox{there exists a $v_s$ to (\ref{helmholtz-s}) such that $T_sv_s=\psi$}
 \}.
\enn
Then it is clear that (\ref{helmholtz-p}), (\ref{helmholtz-s}) and
(\ref{T}) are uniquely solvable for every $\psi\in \mathcal{P}$ (resp. $\psi\in \mathcal{S}$), provided $\omega^2$ is not an eigenvalue of the operator $\mathcal{L}$ with the traction-free boundary condition.
\begin{defn}\label{def:cloak}
The layer of elastic medium $\{\Omega\backslash \overline{D}_{1/2}; \mathcal{C}^{(c)}, \rho^{(c)}\}$ is said to be a full elastic cloak if $\Lambda_{\mathcal{C},\rho}(\psi)=\Lambda_0(\psi)$ for all $\psi\in H^{-1/2}(\partial \Omega)^N$;
 it is called a compressional elastic cloak if $\Lambda_0^{(p)}(\psi)=\Lambda_{\mathcal{C},\rho}(\psi)$ for all $\psi\in\mathcal{P}$; and
it is called a  shear elastic cloak if $\Lambda_0^{(s)}(\psi)=\Lambda_{\mathcal{C},\rho}(\psi)$ for all $\psi\in\mathcal{S}$.
\end{defn}
{ We would like to emphasize that the shear and pressure waves are inherently coupled in the Lam\'e system and that an incident pure shear or pressure wave would incite the two kind of waves simultaneously in general.}
 An inverse problem of significant importance arising in practical applications is to infer information of the interior object $\{\Omega; \mathcal{C}, \rho\}$ by knowledge of the exterior elastic wave measurements. The boundary NtD map $\Lambda_{\mathcal{C},\rho}$ encodes the exterior measurements that one can obtain. We refer to \cite{AK,Amm5, CGK01,HLLS,HLL,HS,Sini,HH,HKJ,GS,NKU,NU1,NU2,NU3} for the theoretical unique identifiability results and numerical reconstruction algorithms developed for these inverse problems. According to Definition~\ref{def:cloak}, the cloaking layer $\{\Omega\backslash \overline{D}_{1/2}; \mathcal{C}^{(c)}, \rho^{(c)}\}$ makes itself and the elastic object $\{D_{1/2}; \mathcal{C}^{(a)}, \rho^{(a)}\}$ undetectable by the exterior elastic wave measurements.

In this paper we focus on the design of full elastic cloaks. In what follows, we show that the entire elastic waves diffracted by $\{D_{1/2}; \mathcal{C}^{(a)}, \rho^{(a)}\}$  can be cloaked if and only if both the compressional and shear waves can be cloaked.
\begin{lem}
Let $\{\Omega;\mathcal{C}, \rho\}$ be an elastic cloak as described above, which is assumed to be regular. Then, $\Lambda_{\mathcal{C},\rho}=\Lambda_0$ if and only if $\Lambda_{\mathcal{C},\rho}=\Lambda_0^{(p)}$ and $\Lambda_{\mathcal{C},\rho}=\Lambda_0^{(s)}$.
\end{lem}
\begin{proof}
The necessity follows directly from the fact that the equations in (\ref{helmholtz-p}) and (\ref{helmholtz-s}) can be reformulated as the Navier equation (\ref{Lame-0}).

Next, we prove the sufficiency. Let $v$ and $u$ solve the boundary value problems (\ref{Lame-0}) and (\ref{I}), respectively. The function $v$ can be decomposed as $v=v_p+v_s$ with $v_p, v_s\in H^1(\Omega)^N$ given by (\ref{kp}) and (\ref{ks}), respectively. Set $\psi_p:=Tv_p=T_pv_p$ and $\psi_s:=Tv_s=T_sv_s$ in $\Omega$. Then $\psi_p+\psi_s=Tv=\psi$ on $\partial\Omega$.
Consider the boundary value problems
\ben
\nabla\cdot (\mathcal{C}: \nabla u_p)+\omega^2\rho u_p=0\quad\mbox{in}\quad \Omega,\quad \mathcal{N}_{\mathcal{C}}u_p=\psi_p\quad\mbox{on}\quad\partial \Omega, \\
\nabla\cdot (\mathcal{C}: \nabla u_s)+\omega^2\rho u_s=0\quad\mbox{in}\quad \Omega,\quad \mathcal{N}_{\mathcal{C}}u_s=\psi_s\quad\mbox{on}\quad\partial \Omega.
\enn
By uniqueness of solutions to (\ref{I}), we have $u=u_p+u_s$. On the other hand, it follows from the assumptions $\Lambda_0^{(p)}=\Lambda_{\mathcal{C},\rho}$ and
$\Lambda_0^{(s)}=\Lambda_{\mathcal{C},\rho}$ that $u_p=v_p$ and $u_s=v_s$ on $\partial\Omega$. Therefore,
\ben
\Lambda_0\psi=v=v_p+v_s=u_p+u_s=u=\Lambda_{\mathcal{C},\rho}\psi\quad\mbox{on}\quad \partial \Omega.
\enn

The proof is complete.
\end{proof}

In the rest of this section, using the transformation-elastodynamics approach based on Lemma~\ref{Lem:change}, we present the blow-up-a-point construction of an ideal full elastic cloak. This elastic cloak has been studied in the physics and engineering literature (cf. \cite{BGM,DG,DG03}), and we shall focus on analyzing the singular structure from a mathematical point of view.  Henceforth, we denote by $B_R$ the central ball of radius $R>0$, and $S_R$ the boundary of $B_R$,
i.e., $S_R=\{x: |x|=R\}$. We take $\Omega=B_2$ and $D_{1/2}=B_1$, following the notations introduced earlier in this section. Let $\{B_2; \mathcal{C}^{(0)}, 1\}$ be the isotropic homogeneous free space, and let $\Lambda_0=\Lambda_{\mathcal{C}^{(0)},1}$ on $S_2$ be the free NtD boundary operator. Consider the transformation
\be\label{F}
F:\left\{\begin{array}{lll}
B_2\backslash\{0\}&\rightarrow& B_2\backslash\overline{B}_1,\\
x&\rightarrow& y=(1+\frac{|x|}{2})\frac{x}{|x|}.
\end{array}\right.
\en
The transform $F$ blows up the origin in the reference space to $B_1$ while maps $B_2\backslash\{0\}$ to $B_2\backslash\overline{B}_1$
 and
keeps the sphere $S_2$ fixed. Using the transformation $F$, the reference medium in $B_2\backslash\{0\}$ is
then push-forwarded to form the transformation medium in $\{B_2\backslash\overline{B}_1; \mathcal{C}^{(c)}, \rho^{(c)}\}$ as follows
\be\label{eqq:31}
\mathcal{C}^{(c)}(y):=F_*(\mathcal{C}^{(0)})(x)|_{x=F^{-1}(y)},\quad \rho^{(c)}(y)=F_*(1)(x)|_{x=F^{-1}(y)},\quad y\in B_2\backslash\overline{B}_1.
\en
Let us consider the boundary value problem (\ref{I}) in $\Omega=B_2$ with
\be\label{eqq:32}
\{B_2; \mathcal{C}, \rho\}=\left\{\begin{array}{lll}
 \{B_2\backslash \overline{B}_{1}; \mathcal{C}^{(c)}, \rho^{(c)}\}&&\mbox{in}\quad B_2\backslash\overline{B}_{1},\\
 \{B_{1}; \mathcal{C}^{(a)}, \rho^{(a)}\}  &&\mbox{in}\quad B_{1}.
\end{array}\right.
\en which defines the NtD map $\Lambda_{\mathcal{C},\rho}$ on $S_2$. By Lemma \ref{Lem:change}, one may infer that
$\Lambda_{\mathcal{C},\rho}=\widetilde{\Lambda}_0$ on $S_2$, where $\widetilde{\Lambda}_0$ is the NtD map associated with the elastic configuration $\{B_2\backslash\{0\}; \mathcal{C}^{(0)}, 1\}$. Noting that the inhomogeneity of the elastic medium $\{B_2\backslash\{0\}; \mathcal{C}^{(0)}, 1\}$ is supported in a singular point, one may infer that $\widetilde{\Lambda}_0=\Lambda_0$, which in turn implies that
$\Lambda_{\mathcal{C},\rho}={\Lambda}_0$. That is, the construction \eqref{eqq:32} yields an ideal full elastic cloak. However, the above argument is rather heuristic and intuitive. Indeed, we shall show that the cloaking elastic medium parameters $\mathcal{C}^{(c)}$ and $\rho^{(c)}$ possess singularities, which make the attempt to rigorous justify the ideal elastic cloak highly nontrivial; see \cite{GKLU,LZ1} for the relevant discussions on the singular optical cloaking of acoustic and electromagnetic waves.

Next, let us determine the explicit expressions of the material parameters for the cloaking medium in \eqref{eqq:31}.
First, one can easily obtain that the Jacobian matrix of $F$ in \eqref{F} and its determinant are given as follows:
\ben
M(y)&=&\frac{r}{2(r-1)}(\textbf{\mbox{I}}-\hat{y}\otimes\hat{y})+\frac{1}{2}\hat{y}\otimes\hat{y},\quad \hat{y}:=y/r,\,r=|y|,\\
\mbox{det}(M)&=&\left\{\begin{array}{lll}
\frac{r}{4(r-1)},&&\mbox{if}\quad N=2,\\
\frac{r^2}{8(r-1)^2},&&\mbox{if}\quad N=3.
\end{array}\right.
\enn
Hence, by Lemma \ref{Lem:change}, the push-forwarded elastic tensor and density in $B_2\backslash\overline{B}_1$ are given by
\be\label{eq:product}
\mathcal{C}^{(c)}(y)=\frac{M(y)\diamond \mathcal{C}^{(0)}(x)|_{x=F^{-1}(y)}\diamond M(y)^\top}{\mbox{det}(M)},\quad
\rho^{(c)}(y)=[\mbox{det}(M)]^{-1},
\en where the operator $\diamond$ denotes the multiplication between a matrix and a fourth-rank tensor. More precisely, in view of the definition of $F_*$ in Lemma \ref{Lem:change}, we have for $\mathcal{C}^{(c)}=(C^{(c)}_{ijkl})_{i,j,k,l=1}^N$
\be\label{eq:productC}
\begin{pmatrix}
C^{(c)}_{i1k1} &C^{(c)}_{i1k2}& \cdots & C^{(c)}_{i1kN}\\
C^{(c)}_{i2k1}& C^{(c)}_{i2k2}& \cdots & C^{(c)}_{i2kN}\\
\vdots & \vdots & \cdots & \vdots \\
C^{(c)}_{iNk1}& C^{(c)}_{iNk2}& \cdots & C^{(c)}_{iNkN}
\end{pmatrix}= M\begin{pmatrix}
C^{(0)}_{i1k1} &C^{(0)}_{i1k2}& \cdots & C^{(0)}_{i1kN}\\
C^{(0)}_{i2k1}& C^{(0)}_{i2k2}& \cdots & C^{(0)}_{i2kN}\\
\vdots & \vdots & \cdots & \vdots \\
C^{(0)}_{iNk1}& C^{(0)}_{iNk2}& \cdots & C^{(0)}_{iNkN}
\end{pmatrix}M^\top/\mbox{det}(M)
\en for $i,k=1,2,\cdots,N$.
 Writing the fourth-rank tensor $\mathcal{C}^{(0)}$ in the tensor product form $(C^{(0)}_{ijkl})_{i,j,k,l=1}^N=(C^{(0)}_{i\cdot k\cdot})_{i,k=1}^N\otimes (C^{(0)}_{\cdot j\cdot l})_{j,l=1}^N$, with the second-rank tensors $(C^{(0)}_{i\cdot k\cdot})_{i,k=1}^N$ and $(C^{(0)}_{\cdot j\cdot l})_{j,l=1}^N$, then the multiplication operator $\diamond$ in the first relation of (\ref{eq:product}) is understood as the two-mode tensor-matrix product in the sense of (\ref{eq:productC}). It is easily seen that the push-forwarded density $\rho^{(c)}$ vanishes on the inner boundary of the cloaking device, namely $S_1$, in $\R^N$. To see the singularity of $\mathcal{C}^{(c)}$, we insert the expression of $M$ into $\mathcal{C}^{(c)}$
\be\no
\mathcal{C}^{(c)}(y)&=&\frac{r^2}{4(r-1)^2}(\textbf{\mbox{I}}-\hat{y}\otimes\hat{y})\diamond \mathcal{C}^{(0)} \diamond
(\textbf{\mbox{I}}-\hat{y}\otimes\hat{y})\;[\mbox{det}(M)]^{-1}\\ \no
&&+\frac{r}{4(r-1)}(\textbf{\mbox{I}}-\hat{y}\otimes\hat{y})\diamond \mathcal{C}^{(0)} \diamond (\hat{y}\otimes\hat{y})\;[\mbox{det}(M)]^{-1}\\ \no
&&+\frac{r}{4(r-1)}(\hat{y}\otimes\hat{y})\diamond \mathcal{C}^{(0)} \diamond(\textbf{\mbox{I}}-\hat{y}\otimes\hat{y})\;[\mbox{det}(M)]^{-1}\\ \label{transformedC}
&&+\frac{1}{4}(\hat{y}\otimes\hat{y})\diamond \mathcal{C}^{(0)} \diamond (\hat{y}\otimes\hat{y})\;[\mbox{det}(M)]^{-1}.
\en
Clearly, in two dimensions, the item in the first line in (\ref{transformedC}) has a singularity of the form $1/(r-1)$ as $r\rightarrow 1$, while the item in the fourth line vanishes on $S_1$.
The 3D spherical cloaks obtained by blowing up a single point turn out to be less singular than the 2D one, since there are no unbounded entries in the transformed elasticity tensor. Using the relations
  \ben
  (\hat{y}\otimes\hat{y}) (\hat{y}\otimes\hat{y})=(\hat{y}\otimes\hat{y}),\quad  (\textbf{I}-\hat{y}\otimes\hat{y}) (\hat{y}\otimes\hat{y})=0,
  \enn
one can deduce from (\ref{transformedC}) that
\ben
\mathcal{C}^{(c)}(y):(\hat{y}\otimes\hat{y})=\frac{1}{4 \det (M)}\,\mathcal{C}^{(0)}(y):(\hat{y}\otimes\hat{y})\rightarrow 0.
\enn as $|y|\rightarrow 1$.
This implies that the tensor
$\mathcal{C}^{(c)}$ does not satisfy the uniform Legendre ellipticity condition \eqref{LH} in $B_2$.

We have calculated the cloaking elastic medium parameters in the Cartesian coordinates using the identity (\ref{eq:productC}). The derivation of the cloaking medium tensor in the 2D polar coordinates $(r, \theta)$  or 3D spherical coordinates $(r,\theta,\varphi)$ can be proceeded as following.
 Noting the symmetric matrix $\hat{y}\otimes\hat{y}$ maps $y$ to its radial direction,
 one can see that the Jacobian matrix $M$ in the polar or spherical coordinates is of the form
 \be\label{M}
 M=\begin{pmatrix}
 1/2 & 0 \\
 0 & \frac{r}{2(r-1)}\,\textbf{I}_{N-1}
 \end{pmatrix}
 \en
 where $\textbf{I}_{N-1}$ denotes the $(N-1)\times (N-1)$ identify matrix.
Here we have employed the following conventional correspondence between indexes: $1\mapsto r$, $2\mapsto\theta$, $3\mapsto\varphi$ in $\R^N$, $N=2,3$. Recalling the \emph{Voigt} notation for tensor indices,
\ben
11\mapsto 1,\;22 \mapsto 2,\; 33 \mapsto 3,\; 23, 32 \mapsto 4,\;
13, 31 \mapsto 5,\; 12,21 \mapsto 6,
\enn we may write the elasticity tensor (\ref{C}) as
\be\label{CV}
C_{\alpha\beta}=\begin{pmatrix}
\lambda+2\mu &\lambda      &\lambda & 0&0 &0  \\
\lambda     & \lambda+2\mu &\lambda & 0&0 & 0 \\
\lambda     &\lambda       &    \lambda+2\mu &0 &0 &0 \\
        0    &      0       &       0           & \mu &0 & 0\\
        0    &      0       &   0              & 0& \mu &0  \\
         0   &   0          &       0           & 0& 0& \mu

\end{pmatrix},\quad
C_{\alpha\beta}=\begin{pmatrix}
\lambda+2\mu &\lambda      & 0\\
\lambda     & \lambda+2\mu & 0\\
0 & 0 & \mu
\end{pmatrix}\en
in three and two dimensions, respectively, with $\alpha,\beta=1,2,\cdots,N$. Using the polar coordinates in 2D, one can deduce from (\ref{eq:productC}), (\ref{M}) and (\ref{CV}) the transformed elasticity tensor $\mathcal{C}^{(c)}$ with eight nontrivial entries (see also \cite{BGM}):
\ben
&&C^{(c)}_{rrrr}=(\lambda+2\mu)(r-1)/r,\quad C^{(c)}_{\theta r \theta r}=\mu r/(r-1),\\
&&C^{(c)}_{r r\theta\theta}=C^{(c)}_{\theta \theta r r}=\lambda,\quad
C^{(c)}_{r\theta\theta r}=C^{(c)}_{\theta r r\theta}=\mu, \\
&& C^{(c)}_{r\theta r\theta}=\mu (r-1)/r,\quad C^{(c)}_{\theta\theta \theta\theta}= (\lambda+2\mu)r/(r-1).
\enn Physically, the vanishing of $C^{(c)}_{rrrr}$, $C^{(c)}_{r\theta r\theta}$ and the singularity of $C^{(c)}_{\theta\theta \theta\theta}$, $C^{(c)}_{\theta r \theta r}$ on $S_1$
imply an infinite velocity of the pressure and shear waves propagating along the inner side of the cloaking interface. The stress tensor for the 3D spherical cloak turns out to have 21 nontrivial entries in $B_2\backslash\overline{B}_1$ with ten of them vanishing on the sphere $S_1$; we refer to \cite{DG} for detailed discussions.

\section{Regularized blowup construction and cloak-busting inclusions}\label{sect:4}

It is seen from our earlier discussion that an elastic cloak can be obtained by using the transformation-elastodynamics approach through a blowup transformation. However, as shown in the last section, the blow-up-a-point would produce singular cloaking medium parameters, which pose server difficulties not only to the corresponding mathematical analysis but also to the practical realization. The singular cloaking medium comes from the use of the singular blowup transformation \eqref{F}. In order to avoid the singular structure, it is nature to regularize the singular blow-up-a-point transformation $F$ as follows. Let $h\in\mathbb{R}_+$ be a sufficiently small regularization parameter, and consider the transformation $F_h: B_2\backslash \overline{B}_h\rightarrow B_2\backslash\overline{B}_1$ defined by
 \ben
F_h(x):=\left\{\begin{array}{lll}
\left(\frac{2-2h}{2-h}+\frac{|y|}{2-h}\right)\,\frac{y}{|y|} && \mbox{in}\quad h\leq|y|\leq 2,\\
y/h&& \mbox{in}\quad |y|< h.
\end{array}\right.
\enn
It is easy to verify that $F_h: B_2\rightarrow B_2$ is bi-Lipschitz, orientation-preserving and $F_h|_{\partial B_2}={\mbox{Identity}}$. Moreover, $F_h$ degenerates to the singular transformation $F$ in \eqref{F} as $h\rightarrow +0$. Now, we consider the cloaking construction similar to \eqref{eqq:32} of the form
\be\label{eqq:32h}
\{B_2; \mathcal{C}, \rho\}=\left\{\begin{array}{lll}
 \{B_2\backslash \overline{B}_{1}; \mathcal{C}_h^{(c)}, \rho_h^{(c)}\}&&\mbox{in}\quad B_2\backslash\overline{B}_{1},\\
 \{B_{1}; \mathcal{C}^{(a)}, \rho^{(a)}\}  &&\mbox{in}\quad B_{1},
\end{array}\right.
\en
with the cloaking medium given by
\be\label{eqq:31h}
\mathcal{C}_h^{(c)}(y):=(F_h)_*(\mathcal{C}^{(0)})(x)|_{x=F_h^{-1}(y)},\quad \rho_h^{(c)}(y)=(F_h)_*(1)(x)|_{x=F_h^{-1}(y)},\quad y\in B_2\backslash\overline{B}_1.
\en
We let $\Lambda_{\mathcal{C},\rho}^h$ denote the NtD map associated with the elastic configuration in \eqref{eqq:32h}. Since $F_h$ degenerates to the singular blow-up-a-point transformation as $h\rightarrow +0$, one may expect that $\Lambda_{\mathcal{C},\rho}^h\rightarrow \Lambda_0$ as $h\rightarrow+0$. That is, \eqref{eqq:32h} would produce an approximate elastic cloak, namely a near-cloak. However, in what follows, we shall show that no matter how small $h\in\mathbb{R}_+$ is, there always exists a certain elastic inclusion $\{B_1; \mathcal{C}^{(a)}, \rho^{(a)}\}$ depending on $h$, which defies any attempt to achieve the near-cloak. Indeed, we shall show that for any $h\in\mathbb{R}_+$, there exists a certain elastic inclusion $\{B_1;\mathcal{C}^{(a)}, \rho^{(a)}\}$ such that the corresponding $\Lambda_{\mathcal{C},\rho}^h$ is not even well-defined due to resonance. In doing so, we first note that by using Lemma~\ref{Lem:change}, the NtD map associated with $\{B_2;\mathcal{C},\rho\}$ in \eqref{eqq:32h} is the same as the one associated with the virtual elastic configuration
\be\label{eqq:virtual11}
\{B_2; \tilde{\mathcal{C}}, \tilde{\rho}\}:=(F_h^{-1})_*\{B_2;\mathcal{C},\rho\}=\left\{\begin{array}{lll}
 \{B_2\backslash \overline{B}_{h}; \mathcal{C}^{(0)}, 1\}&&\mbox{in}\quad B_2\backslash\overline{B}_{h},\\
 \{B_{h}; \tilde{\mathcal{C}}^{(a)}, \tilde{\rho}^{(a)}\}  &&\mbox{in}\quad B_{h},
\end{array}\right.
\en
where $\{B_{h}; \tilde{\mathcal{C}}^{(a)}, \tilde{\rho}^{(a)}\}=(F_h^{-1})_*\{B_1; \mathcal{C}^{(a)}, \rho^{(a)}\}$. That is, the NtD map $\Lambda_{\mathcal{C},\rho}^h$ characterizes the boundary effect due to the small inclusion $\{B_{h}; \tilde{\mathcal{C}}^{(a)}, \tilde{\rho}^{(a)}\}$ supported in $B_h$.
Since the target elastic medium $\{B_1;\mathcal{C}^{(a)}, \rho^{(a)}\}$ is arbitrary but regular, we see that the content of the small inclusion $\{B_{h}; \tilde{\mathcal{C}}^{(a)}, \tilde{\rho}^{(a)}\}$ is in principle also arbitrary but regular. Hence, in order to show the failure of the near elastic cloaking construction \eqref{eqq:32h}, it is sufficient to show that for any $h\in\mathbb{R}_+$, there always exists a certain $\{B_{h}; \tilde{\mathcal{C}}^{(a)}, \tilde{\rho}^{(a)}\}$ such that the NtD map $\Lambda_{\tilde{\mathcal{C}}, \tilde{\rho}}$ associated with the elastic configuration $\{B_2; \tilde{\mathcal{C}}, \tilde{\rho}\}$ is not well-defined due to resonance.

We would like to appeal for a bit more general study in two dimensions only by considering the following Lam\'e system:
\be\label{Homogeneous-N}\left\{\begin{array}{lll}
\nabla\cdot (\mathcal{C}^{(0)}: \nabla u_1)+\omega^2\rho_1 u_1=0\quad\mbox{in}\quad r_0<|x|<r_1,\\
\nabla\cdot (\mathcal{C}^{(0)}: \nabla u_0)+\omega^2\rho_0 u_0=0\quad\mbox{in}\quad |x|<r_0,\\
T u_1=0\quad\mbox{on}\quad |x|=r_1,\\
u_1=u_2,\quad T u_1=T u_2\quad\mbox{on}\quad |x|=r_0.
\end{array}\right.
\en
Here, $\rho_1$, $\rho_0$ are two positive constants, $\omega\in \R_+$ is a fixed frequency, and the elastic tensor $\mathcal{C}^{(0)}$ is given by (\ref{C}) with fixed Lam\'e constants $\lambda,\mu$ in $|x|<r_1$.  In what follows we shall verify that, for any $r_0<r_1$, there always exist constant densities $\rho_0, \rho_1>0$ for which the system \eqref{Homogeneous-N} admits non-trivial solutions. This implies that resonance occurs and the boundary NtD map is not well-defined for the system. Clearly, this also indicates the failure of the near-cloaking construction \eqref{eqq:32h} due to the existence of resonant inclusions.

Let $(r,\varphi)$ be the polar coordinates of $x=(x_1,x_2)\in \R^2$.
We look for special solutions to (\ref{Homogeneous-N}) of the form
\ben
u_j=c_j\,\nabla_x(J_0(k_p^{(j)}r)),\quad \nabla_x:=(\partial_{x_1},\partial_{x_2}),\quad k_p^{(j)}:=\omega^2\sqrt{\rho_j/(\lambda+2\mu)},\quad c_j\in \C,
\enn
i.e., $u_j$ consists of spherically-symmetric compressional waves only. Here $J_0$ denotes the Bessel function of order zero.
Similar examples can be constructed for general elastic waves by using the Bessel function of order $n$.
Simple calculations show that
\be\label{uj}
u_j=c_j\,k_p^{(j)}\, J_0^{'}(k_p^{(j)}r)\begin{pmatrix}
\cos\varphi \\ \sin\varphi\end{pmatrix}.
\en
Since $(\Delta+(k_p^{(j)})^2)J_0(k_p^{(j)}r)=0$, one can readily check that $u_j$ satisfy the Navier equations in
(\ref{Homogeneous-N}) and that
\ben
Tu_1=(k_p^{(1)})^2\begin{pmatrix}
\cos\varphi \\ \sin\varphi\end{pmatrix} \left[2\mu J_0^{''}(k_p^{(1)}r_1)-\lambda J_0(k_p^{(1)}r_1)\right]\,c_1\quad\mbox{on}\quad |x|=r_1.
\enn
On the other hand, the transmission conditions on $|x|=r_0$ in (\ref{Homogeneous-N}) are equivalent to (see Lemma \ref{Lem:stress} below)
\be\label{transmission}
u_1=u_2,\quad \frac{u_1}{\partial r}= \frac{u_2}{\partial r}\quad\mbox{on}\quad |x|=r_0,
\en
due to the invariance of the Lam\'e constants on both sides of the interface. Inserting (\ref{uj}) into (\ref{transmission}) yields the linear system
\ben
\begin{pmatrix}
(k_p^{(1)}r_0)\,J_0^{'}(k_p^{(1)}r_0) & -(k_p^{(2)}r_0)\,J_0^{'}(k_p^{(2)}r_0) \\
(k_p^{(1)}r_0)^2\,J_0^{''}(k_p^{(1)}r_0) & -(k_p^{(2)}r_0)^2\,J_0^{''}(k_p^{(2)}r_0)
\end{pmatrix}\begin{pmatrix}
c_1 \\ c_2
\end{pmatrix}=0.
\enn
Introduce the functions
\ben
f(t):=2\mu J_0^{''}(t)-\lambda J_0(t),\quad g(t):=\frac{J_0^{'}(t)}{t\,J_0^{''}(t)}.
\enn
Since there are infinitely many positive zeros of $f$ tending to infinity, we may choose a $\rho_1>0$ such that
$f(k_p^{(1)}r_1)=0$. Set $t_1:=k_p^{(1)}r_0$. Then we can find a $t_2\in\R_+$, $t_2\neq t_1$ such that
\ben
g(t_1)=g(t_2)\quad\mbox{if}\quad J_0^{''}(t_1)\neq 0;\quad
J_0^{''}(t_2)=0\quad\mbox{if}\quad J_0^{''}(t_1)=0.
\enn
Now, the number $\rho_2>0$ is chosen such that the relation $k_p^{(2)}r_0 =t_2$ holds.
With those choices of $\rho_1$ and $\rho_2$, we have the homogeneous Neumann boundary condition $Tu_1=0$ on $|x|=r_1$. Moreover,
due to the vanishing of
the determinant of the matrix,
 one can find a nontrivial solution ($c_1, c_2$) to the above linear system so that the transmission conditions hold true.  Hence we have constructed a non-trivial pair of solutions $(u_1, u_2)$ to
(\ref{Homogeneous-N}) for any fixed $r_0<r_1$ and $\omega\in \R_+$.

Finally, we give the proof of  (\ref{transmission}).

\begin{lem}\label{Lem:stress}
Let $D\subset \Omega$ be a smooth domain in $\R^2$. Assume that $u_1\in H^1(\Omega\backslash\overline{D})^2$, $u_2\in H^1(D)^2$ satisfy the transmission condition
$u_1=u_2$, $Tu_1=Tu_2$ on $\partial D$. Then $\partial_\nu u_1=\partial_{\nu} u_2$ on $\partial D$.
\end{lem}
\begin{proof}
We shall carry out the proof by making use of the definition (\ref{stress-2D}).
Let $\nu=(\nu_1,\nu_2)$, $\tau=(-\nu_2,\nu_1)$ denote the normal and tangential directions on $\partial D$, respectively. Set $u_j=(u_j^{(1)}, u_j^{(2)})^\top$ for $j=1,2.$
Using the formula
\ben
\partial_1 w=\nu_1 \partial_{\nu} w-\nu_2 \partial_{\tau} w,\quad
\partial_2 w=\nu_2 \partial_{\nu} w+\nu_1 \partial_{\tau} w,
\enn
we may separate the normal and tangential derivatives involved in the stress operator. Consequently,
the stress operator $Tu_j$ on $\partial D$ can be rewritten as
\be\no
Tu_j&=&\begin{pmatrix}
\mu+(\lambda+\mu)\nu_1^2  & (\lambda+\mu) \nu_1 \nu_2\\
(\lambda+\mu) \nu_1 \nu_2 & \mu+(\lambda+\mu)\nu_2^2
     \end{pmatrix}\begin{pmatrix}\partial_{\nu} u_j^{(1)}\\ \partial_{\nu} u_j^{(2)}
 \end{pmatrix}\\ \no
&&+ \begin{pmatrix}
-(\lambda+\mu)\nu_1\nu_2  & \lambda\nu_1^2-\mu\nu_2^2\\
-\lambda\nu_2^2+\mu\nu_1^2 & (\lambda+\mu)\nu_1\nu_2
     \end{pmatrix}\begin{pmatrix}\partial_{\tau} u_j^{(1)}\\ \partial_{\tau} u_j^{(2)}
 \end{pmatrix}\\ \label{eq:AB}
&=:& A(\lambda,\mu,\nu)\,\partial_{\nu} u_j + B(\lambda,\mu,\nu)\,\partial_{\tau} u_j.
\en
Set $U=u_1-u_2$.  From (\ref{eq:AB}) and the assumptions $Tu_1=Tu_2$, $u_1=u_2$ on $\partial D$ we see
\be\label{eq:TU}
0=T U=A(\lambda,\mu,\nu)\,\partial_{\nu} U\quad\mbox{on}\quad \partial D.
\en
Direct calculations yield that the determinant of $A(\lambda,\mu,\nu)$ is given by
\ben
\mbox{det}(A)=\mu (\lambda+2\mu)>0.
\enn
Hence, by (\ref{eq:TU}) we obtain $\partial_{\nu} u_1=\partial_{\nu} u_2$ on $\partial D$.

The proof is complete.
\end{proof}

\section{Nearly cloaking the elastic waves}

\subsection{Our near-cloaking scheme}

Through the discussion in Section~\ref{sect:4}, we see that the regularized blow-up-a-small-ball construction fails due to the existence of resonant inclusions, namely the cloak-busting inclusions. We would like to mention that similar phenomena have been observed in regularized optical cloaks; see \cite{BL,BLZ,KOVW,LS}. In order to defeat the resonance, a natural idea is to introduce a certain damping mechanism. This motivates us to develop a near-cloaking scheme by incorporating a suitable lossy layer right between the cloaked region and cloaking layer.

We are in a position to present the proposed near-cloaking scheme. Let $\Omega$ and $D$ be as described in Section~\ref{sect:3}. Let $h\in\mathbb{R}_+$ be a small regularization parameter and let $F_h$ be a bi-Lipschitz and orientation-preserving mapping such that
\ben
F_h: \overline{\Omega}\backslash D_h\rightarrow \overline{\Omega}\backslash D,\quad F_h(\partial \Omega)=\partial \Omega.
\enn
Introduce the mapping
\ben
F(x)=\left\{\begin{array}{lll}
F_h(x)&&\mbox{for}\quad x\in \overline{\Omega}\backslash D_h,\\
x/h&&\mbox{for}\quad x\in  D_h.
\end{array}\right.
\enn
Clearly, $F:\Omega\rightarrow\Omega$ is bi-Lipschitz and orientation-preserving and $F(\partial \Omega)=\partial \Omega$.

Our proposed regularized near-cloaking construction takes the following general form
\be\label{GeneralOmega}
\{\Omega; \mathcal{C}, \rho\}=\left\{\begin{array}{lll}
 \{\Omega\backslash \overline{D}_{1/2}; \mathcal{C}^{(c)}, \rho^{(c)}\}&&\mbox{in}\quad \Omega\backslash\overline{D}_{1/2},\\
 \{D_{1/2}; \mathcal{C}^{(a)}, \rho^{(a)}\}  &&\mbox{in}\quad D_{1/2},
\end{array}\right.
\en
where
\be\label{GeneralOmega2}
\{\Omega\backslash\overline{D}_{1/2}; \mathcal{C}^{(c)}, \rho^{(c)}\}=\left\{\begin{array}{lll}
 \{\Omega\backslash \overline{D}; \mathcal{C}^{(1)}, \rho^{(1)}\}&&\mbox{in}\quad \Omega\backslash\overline{D},\\
 \{D\backslash\overline{D}_{1/2}; \mathcal{C}^{(2)}, \rho^{(2)}\}  &&\mbox{in}\quad D\backslash \overline{D}_{1/2},
\end{array}\right.
\en
with
\be\label{device-1}
\begin{split}
\{\Omega\backslash\overline{D}; \mathcal{C}^{(1)}, \rho^{(1)}\}&=&(F_h)_{*}\{ \Omega\backslash\overline{D}_h; \mathcal{C}^{(0)}, 1  \},\\
\{D\backslash\overline{D}_{1/2}; \mathcal{C}^{(2)}, \rho^{(2)}\}&=&(F_h)_{*}\{ D_h\backslash\overline{D}_{h/2}; \tilde{\mathcal{C}}^{(2)}, \tilde{\rho}^{(2)}  \}.\end{split}
\en
In \eqref{device-1}, the elastic medium in $D_h\backslash\overline{D}_{h/2}$ is given by
\be\label{construction}
\{ D_h\backslash\overline{D}_{h/2}; \tilde{\mathcal{C}}^{(2)}, \tilde{\rho}^{(2)}  \},\quad \tilde{\mathcal{C}}^{(2)}=\gamma\, h^{2+\delta}\,\mathcal{C}^{(0)},\quad \tilde{\rho}^{(2)}=\alpha+i\beta,
\en
where $\alpha,\beta, \gamma$ and $\delta$ are fixed positive constants. Here, we note that in \eqref{construction}, we introduce a critical lossy layer $\{D_h\backslash\overline{D}_{h/2}; \tilde{\mathcal{C}}^{(2)},\tilde{\rho}^{(2)}\}$, wherein $\beta$ is the damping parameter of the elastic medium. We next present the main theorem in assessing the near-cloaking performance of the above proposed construction. Henceforth, for two Banach spaces $\mathscr{X}$ and $\mathscr{Y}$, we let $\mathscr{L}(\mathscr{X},\mathscr{Y})$ denote the Banach space of the linear functionals from $\mathscr{X}$ to $\mathscr{Y}$. Moreover, we let $C$ denote a generic positive constant, which may change in different estimates, but should be clear in the context. Then, we have

\begin{thm}\label{Theorem}
Assume $-\omega^2   $ is not an eigenvalue of the elliptic operator  $\mathcal{L}$ on $\Omega$ with the traction-free boundary condition. Let $\Lambda_{\mathcal{C},\rho}$ be the NtD map corresponding to the elastic configuration \eqref{GeneralOmega}-\eqref{construction}, and let $\Lambda_0$ be the free NtD map for the Lam\'e system. Then there exists a constant $h_0\in\mathbb{R}_+$ such that when $h<h_0$,
\be\label{eq:mainestimate}
||\Lambda_{\mathcal{C},\rho}-\Lambda_0||_{\mathscr{L}(H^{-1/2}(\partial \Omega)^N, H^{1/2}(\partial \Omega)^N)}\leq C\,h^N,
\en
 where $C$ is a positive constant independent of $h$, $\mathcal{C}^{(a)}$, $\rho^{(a)}$ and $\delta$.
\end{thm}

By Theorem \ref{Theorem}, we see the construction (\ref{GeneralOmega})-\eqref{construction} produces a near-cloak within $h^N$ of the ideal cloak in $\R^N$. Moreover, since the estimate \eqref{eq:mainestimate} is independent of the content being cloak, namely $\{D_{1/2}; \mathcal{C}^{(a)}, \rho^{(a)}\}$, it is capable of nearly cloaking an arbitrary target elastic medium. We would like to remark that $\delta\geq 0$ in \eqref{construction} is a free parameter and one may simply choose it to be $0$.

In order to prove Theorem \ref{Theorem}, we first note that by using Lemma~\ref{Lem:change}
\begin{equation}\label{eq:equivalentestimate}
\Lambda_{\mathcal{C},\rho}=\Lambda_{\tilde{\mathcal{C}},\tilde{\rho}},
\end{equation}
where
\be\label{device-2}
\{\Omega; \tilde{\mathcal{C}},\tilde{\rho}\}=(F^{-1})_*\{\Omega; C, \rho\}=\left\{\begin{array}{lll}
\mathcal{C}^{(0)}   &&\mbox{in}\quad \Omega\backslash \overline{D}_h,\\
\tilde{\mathcal{C}}^{(2)}, \tilde{\rho}^{(2)} &&\mbox{in}\quad D_h\backslash \overline{D}_{h/2},\\
\tilde{\mathcal{C}}^{(a)}, \tilde{\rho}^{(a)} &&\mbox{in}\quad D_{h/2},\\
\end{array}\right.
\en
with
\ben
\{D_{h/2}; \tilde{\mathcal{C}}^{(a)},\tilde{\rho}^{(a)}\}=(F^{-1})_*\{D_{1/2}; \mathcal{C}^{(a)}, \rho^{(a)}\}.
\enn
Let $\tilde{u}\in H^1(\Omega)^N$ be the solution to the Lam\'e system associated with the elastic configuration $\{\Omega;\tilde{\mathcal{C}},\tilde{\rho}\}$; that is
\be\label{Lame-2}
\nabla\cdot (\tilde{\mathcal{C}}: \nabla \tilde{u})+\omega^2\tilde{\rho} \tilde{u}=0\quad\mbox{in}\quad \Omega,\quad \mathcal{N}_{\tilde{\mathcal{C}}}\,\tilde{u}=\psi\in H^{-1/2}(\partial\Omega)^N\quad\mbox{on}\quad\partial \Omega,
\en

Let $\tilde{u}:=u\circ F^{-1}$. Then, by Lemma \ref{Lem:change}, $\tilde{u}$ solves
 the boundary value problem
\be\label{Lame-2}
\nabla\cdot (\tilde{\mathcal{C}}: \nabla \tilde{u})+\omega^2\tilde{\rho} \tilde{u}=0\quad\mbox{in}\quad \Omega,\quad \mathcal{N}_{\tilde{\mathcal{C}}}\,\tilde{u}=\psi\quad\mbox{on}\quad\partial \Omega,
\en
and let $u_0\in H^{1}(\Omega)^N$ be the solution in the free space; see \eqref{Lame-0}. By \eqref{eq:equivalentestimate}, we see that Theorem \ref{Theorem} immediately follows from

\begin{thm}\label{Theorem-1}
Assume $-\omega^2$ is not an eigenvalue of the elliptic operator  $\mathcal{L}$ on $\Omega$ with the traction-free boundary condition. Let $\tilde{u}$ and $u_0$ be solutions to (\ref{Lame-2}) and (\ref{Lame-0}), respectively.   Then there exists a constant $h_0\in\mathbb{R}_+$ such that when $h<h_0$,
\be\label{estimate-1}
||\tilde{u}-u_0||_{H^{1/2}(\partial \Omega)^N}\leq  C\,h^N\, ||\psi||_{H^{-1/2}(\partial \Omega)^N},
\en
where $C$ is a positive constant independent of $h$, $\psi$, $\tilde{\mathcal{C}}$, $\tilde{\rho}$ and $\delta$.
\end{thm}

\subsection{Proof of Theorem \ref{Theorem-1}}

Before giving the proof of Theorem~\ref{Theorem-1}, we sketch the general structure of our argument. First, by using a variational argument together with the use of the lossy layer $\{D_h\backslash\overline{D}_{h/2}; \tilde{\mathcal{C}}^{(2)}, \tilde{\rho}^{(2)}\}$, one can control the energy of the elastic wave field in $D_h\backslash\overline{D}_{h/2}$. Next, by a duality argument, we control the trace of the traction of the elastic wave field on $\partial D_h$. In this step, we need derive a critical Sobolev extension result. Then, the study is reduced to estimating the boundary effect on $\partial \Omega$ due to a small elastic inclusion $D_h$ with a prescribed traction trace on $\partial D_h$. We shall make use of a variety of layer potential techniques in this step. Finally, the sharpness of our estimate has been numerically verified and shall be reported in a forthcoming work.

\begin{lem}\label{Lem-1}
The solutions to (\ref{Lame-2}) and (\ref{Lame-0}) satisfy the estimate
\ben
\beta\, \omega^2 \,||\tilde{u}||^2_{L^2(D_h\backslash\overline{D}_{h/2})}\leq C\,||\psi||_{H^{-1/2}(\partial \Omega)^N}||\tilde{u}-u_0||_{H^{1/2}(\partial \Omega)^N},
\enn where $C$ a is positive constant depending only on $\Omega$.
\end{lem}
\begin{proof}
Multiplying $\overline{\tilde{u}}$ to (\ref{Lame-2}) and integrating by parts yield
\ben
-\int_{\Omega} (\tilde{\mathcal{C}}:\nabla \tilde{u}):\nabla \overline{\tilde{u}}\; dx+\omega^2\int_{\Omega}\tilde{\rho} |\tilde{u}|^2\, dx
&=&-\int_{\partial \Omega} [(\tilde{\mathcal{C}}:\nabla \tilde{u})\cdot \nu]\,\cdot \overline{\tilde{u}}\;ds\\
&=&-\int_{\partial \Omega} \psi\cdot \overline{\tilde{u}}\;ds.
\enn
Similarly,
\ben
-\int_{\Omega} (\mathcal{C}^{(0)}:\nabla u_0):\nabla \overline{u}_0\; dx+\omega^2\int_{\Omega}   |u_0|^2\, dx
=-\int_{\partial \Omega} \psi\cdot \overline{u}_0\;ds.
\enn
Taking the imaginary parts of the above two identities and making use of the definition of $\tilde{\rho}^{(2)}$ in (\ref{construction}), we arrive at
\be\label{eq:3}
\int_{D_{h/2}} \ima(\tilde{\rho}^{(a)}) |\tilde{u}|^2 dx+\beta \omega^2\int_{D_h\backslash D_{h/2}} |\tilde{u}|^2 dx=-\ima \int_{\partial \Omega}\psi\cdot (\overline{\tilde{u}}-\overline{u}_0)\,ds.
\en Since $\ima(\tilde{\rho}^{(a)})\geq0$, Lemma \ref{Lem-1} follows easily from (\ref{eq:3}).

The proof is complete.
\end{proof}

In what follows, we employ the notation $T^\pm \tilde{u}$ to denote the traction operators on $\partial D_h$ when limits are taken from outside and inside of $D_h$, respectively. For simplicity we write $\Psi^\pm(x)=T^\pm \tilde{u}(hx)$ for $x\in \partial D$.
\begin{lem}\label{Lem-2}
Let $\tilde{u}$ and $u_0$ be solutions to (\ref{Lame-2}) and (\ref{Lame-0}), respectively. We have the estimates
\ben
&& ||\Psi^-||^2_{H^{-3/2}(\partial D)^N}\\
&\leq& C\frac{(\gamma+\sqrt{\alpha^2+\beta^2}h^{-\delta}\omega^2)^2}{\beta \gamma^2\omega^2}h^{-N-2}\, ||\tilde{u}-u_0||_{H^{1/2}(\partial \Omega)^N}\,||\psi||_{H^{-1/2}(\partial \Omega)^N},\\
&& ||\Psi^+||^2_{H^{-3/2}(\partial D)^N} \\
&\leq & C\; \frac{(\gamma+\sqrt{\alpha^2+\beta^2}h^{-\delta}\omega^2)^2}{\beta\omega^2}  h^{2(1+\delta)-N}\, ||\tilde{u}-u_0||_{H^{1/2}(\partial \Omega)^N}\,||\psi||_{H^{-1/2}(\partial \Omega)^N},
\enn
where $C$ is a positive constant depending only on $D$ and $\Omega$ but independent of $h$ and $\psi$.
\end{lem}
\begin{proof} By the definition of the norm $||\cdot||_{H^{-3/2}(\partial D)}$,
\ben
||\Psi||_{H^{-3/2}(\partial D)^N}=\sup_{||\phi||_{H^{3/2}(\partial D)^N}=1}\big| \int_{\partial D} \Psi(x)\cdot\phi(x)\,ds\big|.
\enn
For any $\phi\in H^{3/2}(\partial D)^N$, there exists  $w\in H^2(D)^N$ such that (see Lemma \ref{Lem:extension})

(i) $w=\phi$ on $\partial D$ and $T w=0$ on $\partial D$;

(ii) $||w||_{H^2(D)^N}\leq C\,||\phi||_{H^{3/2}(\partial D)^N}$;

(iii) $w=0$ in $D_{1/2}$.

Then we have
\be\label{eq:4}
\int_{\partial D} \Psi^-(x)\cdot \phi(x)\,ds=\int_{\partial D} T^-\tilde{u}(hx)\cdot \phi(x)\,ds=\int_{\partial D} T^-\tilde{u}(hx)\cdot w(x)\,ds.
\en For $y\in D_h$, write $x=y/h \in D$. Set $v(x):=\tilde{u}(hx)=\tilde{u}(y)$ for $x\in D$.
By the definitions of $\tilde{\mathcal{C}}^{(2)}$ and $\tilde{\rho}^{(2)}$ (see (\ref{construction})), we know
\be\label{eq:7}
\gamma\, h^{2+\delta}\,\mathcal{L} \tilde{u}(y)+ \omega^2 (\alpha+i\beta)\tilde{u}(y)=0\quad \mbox{in}\quad D_h\backslash\overline{D}_{h/2}.
\en
Direct calculations show that
\be\label{eq:5}
\gamma\, h^{\delta}\,\mathcal{L} v(x)+ \omega^2 (\alpha+i\beta)v(x)=0\quad \mbox{in}\quad D\backslash\overline{D}_{1/2}.
\en
Using the fact $Tw=0$ on $\partial D$, it is seen from (\ref{eq:4}) that
\ben
\int_{\partial D} \Psi^-(x)\cdot \phi(x)\,ds&=&h^{-1}\int_{\partial D} T^-v(x)\cdot w(x)\,ds\\
&=&
h^{-1}\int_{\partial D} T^-v(x)\cdot w(x)-v(x)\cdot T^-w(x) \,ds\\
&=&
h^{-1}\int_{D\backslash D_{1/2}} \mathcal{L}v\cdot w- \mathcal{L}w\cdot v\,dx
\enn where the third equality follows from Betti's formula and the fact that $w=0$ in $D_{1/2}$.
Recalling (\ref{eq:5}) and applying the Cauchy-Schwarz inequality yield
\be\no
\big|\int_{\partial D} \Psi^-(x)\cdot \phi(x)\,ds\big| &\leq &h^{-1-\delta}\, \sqrt{\alpha^2+\beta^2}\,\omega^2 \gamma^{-1}\, ||v||_{L^2(D\backslash\overline{D}_{1/2})^N} \,||w||_{L^2(D)^N}\\ \label{eq:6}
&&+h^{-1} ||v||_{L^2(D\backslash\overline{D}_{1/2})^N} \,||\mathcal{L} w||_{L^2(D)^N}.
\en
In view of the relations
\ben
&||v||_{L^2(D\backslash\overline{D}_{1/2})^N}=||\tilde{u}(h\,\cdot\,)||_{L^2(D\backslash\overline{D}_{1/2})^N}
=h^{-N/2}||\tilde{u}||_{L^2(D_h\backslash\overline{D}_{h/2})^N},\\
& ||\mathcal{L} w||_{L^2(D)^N}+ || w||_{L^2(D)^N}\leq C \,||\phi||_{H^{3/2}(\partial D)^N},
\enn we derive from (\ref{eq:6}) that
\ben
&&\big|\int_{\partial D} \Psi^-(x)\cdot \phi(x)\,ds\big|\\ &\leq&  C\,h^{-N/2-1}\left(1+ \sqrt{\alpha^2+\beta^2}\,\omega^2 \gamma^{-1}\,h^{-\delta} \right)\, ||\tilde{u}||_{L^2(D_h\backslash\overline{D}_{h/2})^N}\,||\phi||_{H^{3/2}(\partial D)^N}.
\enn
This implies that
\be\label{eq:espsi}
||\Psi^-||_{H^{-3/2}(\partial D)^N}&\leq&  C\,h^{-N/2-1}\left(1+ \sqrt{\alpha^2+\beta^2}\,\omega^2 \gamma^{-1}\,h^{-\delta} \right)\, ||\tilde{u}||_{L^2(D_h\backslash\overline{D}_{h/2})^N},
\en
which together with Lemma \ref{Lem-1} leads to the first assertion of Lemma \ref{Lem-2}.
By (\ref{eq:7}) and the transmission conditions on $\partial D_h$, we have
\ben
T^+\tilde{u}=\gamma\,h^{2+\delta}T^-\tilde{u}\quad\mbox{on}\quad \partial D_h.
\enn Hence, $\Psi^+=\gamma\,h^{2+\delta}\Psi^-$ on $\partial D$. Combining this with the estimate of $||\Psi^-||_{H^{-3/2}(\partial D)^N}$ in \eqref{eq:espsi} proves the second assertion of Lemma \ref{Lem-2}.

The proof is complete.
\end{proof}

We derive the following Sobolev extension result which has been used in the proof of Lemma \ref{Lem-2}.

\begin{lem}\label{Lem:extension}
For any $\phi\in H^{3/2}(\partial D)^N$, there exists  $w\in H^2(D)^N$ such that
\begin{itemize}
\item[(i)] $w=\phi$ on $\partial D$ and $T w=0$ on $\partial D$;
\item[(ii)] $||w||_{H^2(D)^N}\leq C\,||\phi||_{H^{3/2}(\partial D)^N}$;
\item[(iii)] $w=0$ in $D_{1/2}$.
\end{itemize}
\end{lem}
\begin{proof}
For  $\psi\in H^{1/2}(\partial D)^N$, one can clearly find $w_1\in H^2(D)^N$ such that
\ben
w_1=0\quad \mbox{in}\; D_{1/2}, \quad \partial_{\nu}w_1=\psi\quad\mbox{on}\; \partial D,\quad ||w_1||_{H^2(D)^N}\leq C\,||\psi||_{H^{1/2}(\partial D)^N}.
\enn
By \cite[Theorem 14.1]{Wloka}, there exists $w_2\in H^2(D)^N$ such that
\begin{itemize}
\item[(i)] $w_2=\phi-w_1$ on $\partial D$ and $\partial_{\nu}w_2=0$ on $\partial D$;
\item[(ii)] $||w_2||_{H^2(D)^N}\leq C\,||\phi-w_1||_{H^{3/2}(\partial D)^N}$;
\item[(iii)] $w_2=0$ in $D_{1/2}$.
\end{itemize}
Hence, the sum $w:=w_1+w_2$ satisfies
\begin{itemize}
\item[(a)] $w=\phi$ on $\partial D$ and $\partial_{\nu}w=\psi$ on $\partial D$;
\item[(b)] $||w||_{H^2(D)^N}\leq C (||\phi||_{H^{3/2}(\partial D)^N}+ ||\psi||_{H^{1/2}(\partial D)^N})$;
\item[(c)] $w=0$ in $D_{1/2}$.
\end{itemize}
In order to conclude the proof of the lemma, it is sufficient to determine a $\psi=\psi(\phi)\in H^{1/2}(\partial D)^N$ such that
\be\label{eq:Tw}
T w=0\quad\mbox{on}\quad \partial D\quad\mbox{and}\quad
||\psi||_{H^{1/2}(\partial D)^N}\leq C\,||\phi||_{H^{3/2}(\partial D)^N}.
\en
In two dimensions, we recall from (\ref{eq:AB}) that the stress operator can be decomposed into
\ben
Tw=A(\lambda,\mu,\nu)\,\partial_{\nu} w + B(\lambda,\mu,\nu)\,\partial_{\tau} w\quad \mbox{on}\quad \partial D.
\enn
In particular, the matrix $A$ is invertible,  $B$ is bounded and the tangential derivative $\partial_{\tau}w\in H^{1/2}(\partial D)^N$  is uniquely determined by $w=\phi$ on $\partial D$. Hence,
choosing $\psi:=-A^{-1}B \partial_\tau w\in H^{1/2}(\partial D)^N$, we see that the relations in (\ref{eq:Tw}) are both fulfilled.

We next consider the 3D case and we need the following identity
\be\label{eq:Grad}
\grad \varphi=\Grad \varphi+\frac{\partial \varphi}{\partial \nu}\,\nu,\quad \nu=(\nu_1,\nu_2,\nu_3),
\en
where $\Grad(\cdot) $ denotes the surface gradient of a scalar function on $\partial D$. Write $w=(w_1,w_2,w_3)$ and $\Grad w_i=([\Grad w_j]^1, [\Grad w_j]^2, [\Grad w_j]^3)$. By (\ref{eq:Grad}) and the definition (\ref{stress-3D}), one can represent the three dimensional stress operator as
\ben
Tw=A(\lambda,\mu,\nu)\,\partial_{\nu} w + \textbf{B}(\lambda,\mu,\nu, \Grad w),
\enn
where
\ben
\textbf{B}&:=&\lambda\nu\;\left([\Grad w_1]^1+[\Grad w_2]^2+[\Grad w_3]^3\right)\\  &&\!\!\!\!+ \mu\, \nu\times \left(
[\Grad w_3]^2-[\Grad w_2]^3, [\Grad w_1]^3-[\Grad w_3]^1, [\Grad w_2]^1-[\Grad w_1]^2\right),\\
A&:=&2\mu +\lambda \nu (\nu\;\cdot\; )+\mu \nu\times(\nu\times\; ) \\
&&=\begin{pmatrix}
\mu+(\lambda+\mu)\nu_1^2 & (\lambda+\mu)\nu_1\nu_2 & (\lambda+\mu)\nu_1\nu_3\\
(\lambda+\mu)\nu_1\nu_2  & \mu+(\lambda+\mu)\nu_2^2 & (\lambda+\mu)\nu_2\nu_3\\
(\lambda+\mu)\nu_1\nu_3 & (\lambda+\mu)\nu_2\nu_3 & \mu+(\lambda+\mu)\nu_3^2
 \end{pmatrix}.
\enn
It is straightforward to check that
\ben
||\textbf{B}||_{H^{1/2}(\partial D)^N}\leq C\,||w||_{H^{1/2}(\partial D)^N}= C\,||\phi||_{H^{1/2}(\partial D)^N},
\enn since only the surface gradients are involved in $\textbf{B}$. On the other hand, we have
$\mbox{det}(A)=\mu^2(\lambda+2\mu)>0$. Hence, as done in the 2D case, one can take $\psi:=-A^{-1}\textbf{B} \in H^{1/2}(\partial D)^N$.
This verifies (\ref{eq:Tw}) in three dimensions and completes the proof.
\end{proof}

\begin{lem}\label{Lem-3}
Assume that $-\omega^2  $ is not an eigenvalue of the elliptic operator $\mathcal{L}$ on $\Omega$ with the traction-free boundary condition.  Let $u_0 \in H^1(\Omega)^N$ be the solution of (\ref{Lame-0}) with $\psi\in H^{-1/2}(\partial \Omega)^N$. For $h>0$ and $\varphi\in H^{-1/2}(\partial D_{h})^N$, consider the elliptic boundary value problem
\ben \left\{\begin{array}{lll}
\mathcal{L}v+\omega^2   v=0&&\mbox{in}\quad \Omega\backslash\overline{D}_{h},\\
 T v=\varphi&&\mbox{on}\quad \partial D_{h},\\
 Tv=\psi&&\mbox{on}\quad \partial \Omega.\end{array}\right.
\enn
Then there exists $h_0\in\mathbb{R}_+$ such that when $h<h_0$,
\be\label{eq:10}
||v-u_0||_{H^{1/2}(\partial \Omega)^N}\leq C\,\left(h^N\,||\psi||_{H^{-1/2}(\partial \Omega)^N}  + h^{N-1}\,||\varphi(h\,\cdot\,)||_{H^{-3/2}(\partial D)^N} \right),
\en where $C$ is a positive constant independent of $h$, $\varphi$ and $\psi$.
\end{lem}
\begin{proof}
Set $w=v-u_0$ in $\Omega\backslash\overline{D}_h$. Then $w\in H^1(\Omega\backslash\overline{D}_h)^N$ satisfies
\ben \left\{\begin{array}{lll}
\mathcal{L}w+\omega^2   w=0&&\mbox{in}\quad \Omega\backslash\overline{D}_{h},\\
 T w=\varphi-T u_0&&\mbox{on}\quad \partial D_{h},\\
 Tw=0&&\mbox{on}\quad \partial \Omega.\end{array}\right.
\enn
Let $w_1\in H^1(\Omega\backslash\overline{D}_h)^N$ be the unique solution of
\ben \left\{\begin{array}{lll}
\mathcal{L}w_1+\omega^2   w_1=0&&\mbox{in}\quad \Omega\backslash\overline{D}_{h},\\
 T w_1=\varphi&&\mbox{on}\quad \partial D_{h},\\
 Tw_1=0&&\mbox{on}\quad \partial \Omega.\end{array}\right.
\enn
Then $w_2:=w-w_1\in H^1(\Omega\backslash\overline{D}_h)^N$ satisfies
\ben \left\{\begin{array}{lll}
\mathcal{L}w_2+\omega^2   w_2=0&&\mbox{in}\quad \Omega\backslash\overline{D}_{h},\\
 T w_2=T u_0&&\mbox{on}\quad \partial D_{h},\\
 Tw_2=0&&\mbox{on}\quad \partial \Omega.\end{array}\right.
\enn
By Lemma \ref{Au-Le-1} in Section~\ref{sec:small}, we know
\ben
||w_2||_{H^{1/2}(\partial \Omega)^N}\leq C\,h^N\,||\psi||_{H^{-1/2}(\partial \Omega)^N}.
\enn
Hence, in order to prove (\ref{eq:10}) we only need to verify
\be\label{eq:13}
||w_1||_{H^{1/2}(\partial \Omega)^N}\leq C\,h^{N-1}\,||\varphi(h\,\cdot\,)||_{H^{-3/2}(\partial D)^N}.
\en
To that end, we consider the following elastic scattering problem in an unbounded domain:
 \ben
\mathcal{L}\, W+\omega^2   W=0\quad\mbox{in}\quad \R^N\backslash\overline{D}_{h},\quad
 T\, W=\varphi\quad\mbox{on}\quad \partial D_{h},
\enn
where $W\in H^1_{loc}(\R^N\backslash\overline{D}_h)$ is additionally required to satisfy the Kupradze radiation condition (see (\ref{Rads}) below) when $|x|\rightarrow \infty$.
We shall show in Section \ref{sec:small} that (see Lemma \ref{Au-Le-2})
\be\label{eq:12}
||W||_{H^{1/2}(\partial \Omega)^N}+||T W|_{\partial \Omega}||_{C(\partial \Omega)}\leq C\,h^{N-1}\,||\varphi(h\,\cdot\,)||_{H^{-3/2}(\partial D)^N}.
\en
Obviously, the difference $P:=w_1-W\in H^1(\Omega\backslash\overline{D}_h)$
 is the unique solution to
\ben \left\{\begin{array}{lll}
\mathcal{L} P+\omega^2   P=0&&\mbox{in}\quad \Omega\backslash\overline{D}_{h},\\
 T P=0&&\mbox{on}\quad \partial D_{h},\\
 T P=-TW&&\mbox{on}\quad \partial \Omega.\end{array}\right.
\enn Making use of layer potential techniques, one can show that
\be\label{eq:11}
||P||_{H^{1/2}(\partial \Omega)^N}\leq C\,||T W|_{\partial \Omega}||_{C(\partial \Omega)}.
\en
(\ref{eq:11}) can be proved in a completely similar manner to that of Lemma \ref{Au-Le-1} in what follows.  Finally, combining (\ref{eq:12}) and (\ref{eq:11}) yields the estimate (\ref{eq:13}), which completes the proof.
\end{proof}

\begin{proof}[Proof of Theorem \ref{Theorem-1}]
 We set $\varphi=T^+\tilde{u}|_{\partial D_h}$ in Lemma \ref{Lem-2}, so that $v=\tilde{u}$ and $\Psi^+=\varphi(h\,\cdot\,)$. By Lemma \ref{Lem-2}, it holds that
\be\label{eq:8}
||\tilde{u}-u_0||_{H^{1/2}(\partial \Omega)^N}\leq C_1\,\left(h^N\,||\psi||_{H^{-1/2}(\partial \Omega)^N}  + h^{N-1}\,||\Psi^+||_{H^{-3/2}(\partial D)^N} \right).
\en
Recalling from the second assertion of Lemma \ref{Lem-2} that, for sufficiently small $h$,
\be\label{eq:9}
||\Psi^+||_{H^{-3/2}(\partial D)^N}\leq C_2\, h^{(2-N)/2}\,||\tilde{u}-u_0||^{1/2}_{H^{1/2}(\partial \Omega)^N}\,||\psi||^{1/2}_{H^{-1/2}(\partial \Omega)^N}.
\en
Combining (\ref{eq:8}) and (\ref{eq:9}) and applying Young's inequality yield the desired estimate in \eqref{estimate-1}.

The proof is complete.
\end{proof}

\subsection{Estimates on small inclusions}\label{sec:small}
\subsubsection{Layer potentials for the Lam\'e system}

We first recall
the fundamental solution $\Pi(x,y)$ (Green's tensor) to the Navier equation (\ref{Lame-0}) in $\R^N$. Let $G_k(x,y)$ denote the free-space fundamental solution to the scalar Helmholtz equation $(\Delta+k^2)u=0$ in $\R^N$. In three dimensions, it takes the form
\ben
G_k(x,y)=\frac{\exp (ik |x-y|)}{4\pi\,|x-y|},\quad x\neq y,\quad x,y\in \R^3,
\enn
while in two dimensions,
\ben
G_k(x,y)=\frac{i}{4}\,H_0^{(1)}(k|x-y|),\quad x\neq y,\quad x,y\in \R^2,
\enn
where $H_0^{(1)}(\cdot)$ is the Hankel function of the first kind of order zero.
Then the Green's tensor $\Pi(x,y)$ for the Lam\'e system can be represented as
\be\label{Pi}
\Pi^{(\omega)}(x,y)=\frac{1}{\mu}\, G_{k_s}(x,y)\,\textbf{I}+\frac{1}{ \omega^2  }\, \grad_x\,\grad_x^\top\;\left[ G_{k_s}(x,y)-G_{k_p}(x,y) \right],
\en for $x, y\in \R^N$, $x\neq y$, where the compressional and shear wave numbers $k_p$ and $k_s$ are given respectively in (\ref{kp}) and (\ref{ks}), and
$\textbf{I}$ stands for the $N\times N$ identity matrix.

Let $Q$ be a bounded simply connected domain in $\mathbb{R}^N$ with the smooth boundary $\partial Q$.  In our subsequent applications, $Q=D_h$ or $Q=\Omega$.
For surface densities $\varphi(x)$ with $x\in \partial Q$,
define the single and double layer potential operators for the Navier equation by
\be\label{SL}
(SL_Q\varphi)(x):=\int_{\partial Q} \Pi(x,y)\varphi(y)ds(y),\quad  x\in \R^N\backslash\partial Q,\;\\ \label{DL}
(DL_{Q}\varphi)(x):=\int_{\partial Q}  \Xi(x,y)\varphi(y)ds(y),\quad\ x\in \R^N\backslash \partial Q,
\en where $\Xi(x,y)$ is a matrix function whose $l$-th column vector is defined as
\ben
 \left[\Xi(x,y)\right]^\top\, \textbf{e}_l:=T_{\nu(y)}\left[\Pi(x,y)\textbf{e}_l\right]=\nu(y)\cdot\left[ \sigma(\Pi(x,y)\,\textbf{e}_l)\right]\quad\mbox{on}\quad\partial Q,
\enn
for $x\neq y,l=1,2,N$.
Here, $\textbf{e}_l$, $1\leq l\leq N$ are the standard Euclidean base vectors in $\mathbb{R}^N$, and $T_{\nu(y)}$ is the stress operator defined in (\ref{stress-2D}) and (\ref{stress-3D}).
We also let
\be\label{S}
(S_Q\varphi)(x):=\int_{\partial Q} \Pi(x,y)\varphi(y)ds(y),\quad  x\in \partial Q,\;\\ \label{K}
(K_{Q}\varphi)(x):=\int_{\partial Q}  \Xi(x,y)\varphi(y)ds(y),\quad\ x\in \partial Q.
\en
Using Taylor series expansion for exponential functions, one can rewrite the matrix $\Pi^{(\omega)}(x,y)$ in 3D as the series (see, e.g., \cite{Amm5})
\be\nonumber
\Pi^{(\omega)}(x,y)&=&\frac{1}{4\pi}\sum_{n=0}^{\infty}\frac{(n+1)(\la+2\mu)+\mu}{\mu (\la+2\mu)}
\frac{(i\omega)^n}{(n+2)n!} |x-y|^{n-1}\, \textbf{I}\\ \label{GO}
&&- \frac{1}{4\pi}\sum_{n=0}^{\infty}\frac{\la+\mu}{\mu (\la+2\mu)}\frac{(i\omega)^n (n-1)}{(n+2)n!} |x-y|^{n-3}\,(x-y)\otimes(x-y),
\en where $x\otimes x:=x^\top x\in \R^{N\times N}$ for $x=(x_1,\cdots, x_N)\in \R^N$.
Letting $x\rightarrow y$, we get
\be\no
 \Pi^{(\omega)}(x,y)&=&\frac{\la+3\mu}{8\pi \mu (\la+2\mu)} \frac{1}{|x-y|}\,\textbf{I}+i\omega\frac{2\la+5\mu}{12\pi \mu (\la+2\mu)}\,\textbf{I}\\ \label{Gomega3}
 &&+\frac{\la+\mu}{8\pi\mu(\la+2\mu)}\frac{1}{|x-y|^3}\,(x-y)\otimes(x-y)
 +o(1)\omega^2.
\en
Taking $\omega\rightarrow +0$ in (\ref{Gomega3}), we obtain
 the fundamental tensor of the Lam\'e system with $\omega=0$ in $\R^3$:
\be\label{Pi0}
\Pi^{(0)}(x,y)=\frac{\lambda+3\mu}{8\pi\mu (\lambda+2\mu)}\frac{1}{|x-y|}\,\textbf{I}+ \frac{\lambda+\mu}{8\pi\mu (\lambda+2\mu)}\frac{1}{|x-y|^3}\, (x-y)\otimes(x-y).
\en
Analogously, in two dimensions we have the expression (see \cite[Chapter 2.2]{Hsiao})
\be\label{G0}
 \Pi^{(0)}(x,y)=\frac{1}{4\pi}\left[-\frac{3\mu+\la}{\mu(2\mu+\la)}\ln |x-y|\,\textbf{I}
+\frac{\mu+\la}{\mu(2\mu+\la)\,|x|^2}\, (x-y)\otimes(x-y)\right].
\en
Similar to the definitions of $SL_Q$, $DL_Q$, $S_Q$, $D_Q$, we define the operators $SL^{(0)}_Q$, $DL^{(0)}_Q$, $S^{(0)}_Q$, $D^{(0)}_Q$ in the same way as (\ref{SL}), (\ref{DL}), (\ref{S}) and (\ref{K}), but with the tensor $\Pi^{(\omega)}(x,y)$ replaced by $\Pi^{(0)}(x,y)$. It is well known that these operators
 all have weakly singular kernels; see, e.g., \cite{Hsiao} and \cite{Kupradze}.

Using the asymptotic behavior of Bessel functions, it has been shown in two dimensions (see e.g.,\cite[Lemma 2.1]{HS})
\be\label{eq:14}
\Pi^{(\omega)}(x,y)=\Pi^{(0)}(x,y) +
\eta\,\textbf{I}\,+\mathcal{O}(|x-y|^2\ln |x-y|)
\en as $x\rightarrow y$, where $\eta$ is a constant given by
\ben
\eta=-\frac{1}{4\pi}\left[ \frac{\la+3\mu}{\mu(\la+2\mu)} (\ln \frac{\omega}{2}+E-\frac{i\pi}{2})+ \frac{\la+\mu}{\mu(\la+2\mu)}-\frac{1}{2}(\frac{\ln \mu}{\mu}+\frac{\ln (\la+2\mu)}{\la+2\mu})\right],
\enn with $E=0.57721\cdots$ being Euler's constant.
From the asymptotic behavior (\ref{eq:14}), it follows that
for $x\in \overline{D}_h$,
\be\label{eq:15}
\int_{D_h} || \Pi^{(\omega)}(x,y) ||_{\max}\,dy=\mathcal{O}(h^2\ln h),\;
\int_{D_h} || \partial_{x_j}\Pi^{(\omega)}(x,y) ||_{\max}\,dy=\mathcal{O}(h),
\en for $j=1,2,3$ as $h\rightarrow 0$,
where $||\cdot||_{\max}$ denotes the maximum norm of a matrix. Analogously, we may deduce from (\ref{Gomega3}) that in 3D,
\be\label{eq:16}
\int_{D_h} || \Pi^{(\omega)}(x,y) ||_{\max}\,dy=\mathcal{O}(h^2),\;
\int_{D_h} || \partial_{x_j}\Pi^{(\omega)}(x,y) ||_{\max}\,dy=\mathcal{O}(h),\; j=1,2,3,
\en
as $h\rightarrow0$. The relations in (\ref{eq:15}) and (\ref{eq:16}) remain valid for all $\omega\geq 0$. The difference $\Pi^{(\omega)}(x,y)-\Pi^{(0)}(x,y)$ is a continuous function in $\R^N\times\R^N$.

\subsubsection{Auxiliary lemmas with estimates on small inclusions}
\begin{lem}\label{Au-Le-1}
Assume that $-\omega^2  $ is not an eigenvalue of the elliptic operator $\mathcal{L}$ on $\Omega$ with the traction-free boundary condition. Let $u_0\in H^1(\Omega)^N$ be the unique solution of (\ref{Lame-0}) with $\psi\in H^{-1/2}(\partial \Omega)^N$. Consider the Lam\'e system
\be\label{eq:18} \left\{\begin{array}{lll}
\mathcal{L}w+\omega^2   w=0&&\mbox{in}\quad \Omega\backslash\overline{D}_{h},\\
 T w=T u_0&&\mbox{on}\quad \partial D_{h},\\
 Tw=0&&\mbox{on}\quad \partial \Omega.\end{array}\right.
\en Then there exists a constant $h_0\in\mathbb{R}_+$ such that for all $h<h_0$, the above Lam\'e system admits a unique solution $w\in H^1(\Omega\backslash\overline{D}_h)^N$. Moreover, there holds the estimate
\be\label{eq:estlem55}
||w||_{H^{1/2}(\partial \Omega)^N}\leq C\,h^N\,||\psi||_{H^{-1/2}(\partial \Omega)^N}
\en
where $C$ is a positive constant independent of $h$ and $\psi$.
\end{lem}

\begin{proof}
For clarity we divide our proof into three steps.

{\bf Step 1.} Show that the function
\ben
V(x):=\int_{\partial D_h} \Pi(x,y)\, Tu_0(y)\,ds(y),\quad x\in \Omega\backslash D_h
\enn satisfies the estimates
\be\label{eq:17}
||V(h\,\cdot\,)||_{C(\partial D)}\leq C\,h||\psi||_{H^{-1/2}(\partial \Omega)^N},\quad
||V||_{C(\partial\Omega)}\leq C\,h^N\,||\psi||_{H^{-1/2}(\partial \Omega)^N}.
\en
We first estimate $V(x)$ for $x\in \partial \Omega$. From Betti's formula, we may rewrite the $j$-th component $V_j$ of $V$ as
\be\label{H}
V_j(x)&=&\int_{\partial D_h} [\Pi^\top(x,y)\textbf{e}_j]\cdot Tu_0(y)\,ds(y)\\ \no
&=&\int_{ D_h} \mathcal{L} u_0(y)\cdot[\Pi^\top(x,y)\textbf{e}_j]dy+\int_{D_h} [\mathcal{C}^{(0)}:\nabla_y (\Pi^\top(x,y)\textbf{e}_j)]: \nabla_y u_0\,dy\\ \no
&=&-\omega^2\int_{ D_h}  u_0(y)\cdot[\Pi^\top(x,y)\textbf{e}_j]dy+\int_{D_h} [\mathcal{C}^{(0)}:\nabla_y (\Pi^\top(x,y)\textbf{e}_j)]: \nabla_y u_0\,dy.
\en
Since $||\Pi(x,y)||_{\max}$ and $||\partial_{y_j}\Pi(x,y)||_{\max}$ are uniformly bounded for all $x\in\partial \Omega$, $y\in \partial D_h$ and for all $j=1,2,3$, we readily derive from (\ref{H}) that
\ben
|V_j(x)|\leq C\,h^N\,\left(\omega^2 ||u_0||_{L^\infty(D_h)^N}+||\nabla u_0||_{L^\infty(D_h)^N}\right)\leq C\,h^N ||\psi||_{H^{-1/2}(\partial \Omega)^N},
\enn for all $j=1,2,3$,
where the last inequality follows from the stability of the boundary value problem (\ref{Lame-2}). This proves the second estimate in (\ref{eq:17}).
The first estimate when $ x\in \partial D_h$ follows straightforwardly from (\ref{H}), the relations in (\ref{eq:15}) and (\ref{eq:16}), together with the fact that both
$||u_0||_{L^\infty(D)^N}$ and $||\nabla u_0||_{L^\infty(D)^N}$ are bounded by $||\psi||_{H^{-1/2}(\partial \Omega)^N}$.

{\bf Step 2.} Set $\phi_1=w|_{\partial D_h}$, $\phi_2=w|_{\partial \Omega}$. In this step, we shall verify
\be\label{eq:19}
||\phi_1||_{L^2(\partial D_h)^N}\leq C\,h^{(N+1)/2}\,||\psi||_{H^{-1/2}(\partial \Omega)^N},\quad
||\phi_2||_{L^2(\partial \Omega)^N}\leq C\,h^{N}\,||\psi||_{H^{-1/2}(\partial \Omega)^N}.
\en
 Again using Betti's formula, we represent the solution $w$ to (\ref{eq:18}) as
\ben
w(x)&=&\int_{\partial\Omega\cup \partial D_h} \left\{\Pi(x,y) T w(y)-\Xi(x,y) w(y)\right\}\,ds(y)\\ \no
&=&-\int_{\partial\Omega}\Xi(x,y) w(y)\,ds(y)-\int_{ \partial D_h} \left\{\Pi(x,y) T u_0(y)-\Xi(x,y) w(y)\right\}\,ds(y)\\
&=&- DL_{\partial \Omega} (\phi_2)(x)+DL_{\partial D_h} (\phi_1)(x)-V(x),
\enn where the function $H$ is defined in Step 1.
Since $Tu_0$ is smooth on $\partial D_h$ and the boundaries of $\Omega$ and $D_h$ are both smooth, by the elliptic regularity $w$ is smooth up to the boundary of $\Omega\backslash D_h$. Hence $\phi_1$ and $\phi_2$ are both smooth functions.
Letting $x$ tend to $\partial \Omega$ and $\partial D_h$, and applying the jump relations for double layer potentials, we have for $\phi_1\in C(\partial D_h)$, $\phi_2\in C(\partial \Omega)$ that
\be\label{eq:20}\left\{\begin{array}{lll}
\frac{1}{2}\phi_1(x)&=&\left(K_{\partial D_h}\phi_1\right)(x)-\left(DL_{\partial \Omega}\phi_2  \right)(x)+V(x),\quad x\in\partial D_h,\\
\frac{1}{2}\phi_2(x)&=&\left(DL_{\partial D_h}\phi_1\right)(x)-\left(K_{\partial \Omega}\phi_2  \right)(x)+V(x),\quad x\in\partial \Omega.
\end{array}\right.
\en
Since $-\omega^2$ is not an eigenvalue of the operator $\mathcal{L}$ on $\Omega$ with the traction-free boundary condition, the operator $\frac{1}{2}I+K_{\partial \Omega}: C(\partial \Omega)\rightarrow C(\partial \Omega)$ is continuously invertible. Thus it follows from (\ref{eq:20}) and (\ref{eq:17})  that
\be\no
||\phi_2||_{C(\partial \Omega)}&\leq& C\, \left(||DL_{\partial D_h}\phi_1||_{C(\partial \Omega)} + ||V  ||_{C(\partial \Omega)}\right)
\\ \label{eq:22}
&\leq& C\,\left(h^{(N-1)/2}\,||\phi_1||_{L^2(\partial D_h)^N}+ h^N\,||\psi||_{H^{-1/2}(\partial \Omega)^N}\right).
\en
Since the $L^2$-norm of $\phi_2$ can also be bounded by the left hand side of (\ref{eq:22}), we only need to verify the first estimate in (\ref{eq:19}). To that end, we rewrite the first equation in (\ref{eq:20}) as
\be\label{eq:23}
\left[\left(\frac{1}{2} I+K_{\partial D}^{(0)}-\mathcal{R}\right) \phi_1(h\,\cdot\,)\right](h x)+\left( DL_{\partial \Omega}\phi_2(\cdot)\right)(hx)=V(hx),\quad x\in D,
\en where the kernel of the operator $\mathcal{R}: L^2(\partial D)^N\rightarrow L^2(\partial D)^N$ is given by the continuous matrix $\Pi^{(0)}-\Pi^{(\omega)}$. Further, it can be straightforwardly checked that
\ben
||\mathcal{R}||_{L^2(\partial D)^N\rightarrow L^2(\partial D)^N}\leq C\,h^{N-1}.
\enn On the other hand, the $L^2$-norm of  $DL_{\partial \Omega}\phi_2(h\,\cdot\,)$ over $\partial D$ can be bounded by the left hand side of (\ref{eq:22}) and that of $V(h\,\cdot\,))$ can be estimated as in the first relation of (\ref{eq:17}).
Hence, by the boundedness of $(\frac{1}{2}I-K^{(0)}_{\partial D})^{-1}: L^2(\partial D)^N\rightarrow L^2(\partial D)^N$, we deduce from (\ref{eq:23}) that
\ben
 ||\phi_1(h\,\cdot\,)||_{L^2(\partial D)^N}\leq C\,h\,||\psi||_{H^{-1/2}(\partial \Omega)^N},
\enn leading to the first relation of (\ref{eq:19}) on $\partial D_h$.

{\bf Step 3.}  By the second equality in (\ref{eq:20}) and the definition of $\phi_2$ in Step 2, one has
\be\label{eq:28}
w(x)=2\left[\left(DL_{\partial D_h}\phi_1\right)(x)-\left(K_{\partial \Omega}\phi_2\right)(x) +V(x)\right],\quad x\in\partial\Omega.
\en
By a similar argument to that for the proof of the second relation in (\ref{eq:17}), one can show that
\be \label{eq:27}
||V||_{H^{1/2}(\partial \Omega)^N}\leq C\, ||V||_{C^1(\partial \Omega)^N}\leq C\,h^N\,||\psi||_{H^{-1/2}(\partial \Omega)^N},
\en Further, using the first estimate in (\ref{eq:19}) yields
\be\label{eq:25}
||DL_{\partial D_h}\varphi_1||_{H^{1/2}(\partial \Omega)^N}\leq C\, ||DL_{\partial D_h}\varphi_1||_{C^1(\partial \Omega)^N}\leq C\,h^N\,||\psi||_{H^{-1/2}(\partial \Omega)^N}.
\en  Since $K_{\partial \Omega}: L^2(\Omega)^N\rightarrow H^1(\Omega)^N$ is bounded, by the second estimate of (\ref{eq:19}) we find
\be\label{eq:26}
||K_{\partial \Omega}\phi_2||_{H^{1/2}(\partial \Omega)^N}\leq C\, h^N\,||\psi||_{H^{-1/2}(\partial \Omega)^N}.
\en
 Combining (\ref{eq:28})-(\ref{eq:26}) yields \eqref{eq:estlem55}.

 The proof is completed.
\end{proof}

Consider the time-harmonic elastic scattering problem from a small cavity $D_h\subset \Omega$. This can be modelled by the following boundary value problem in the exterior of $D_h$: find $W\in H^1_{loc}(\R^N\backslash\overline{D}_h)^N$ such that
 \be\label{Lame-cavity}
\mathcal{L}\, W+\omega^2   W=0\quad\mbox{in}\quad \R^N\backslash\overline{D}_{h},\quad
 T\, W=\varphi\quad\mbox{on}\quad \partial D_{h}.
\en Since $\R^N\backslash\overline{D}_{h}$ is unbounded,
 $W$ is  required to satisfy the Kupradze radiation condition when $|x|\rightarrow \infty$
(see, e.g. \cite{AK}):
\be\label{Rads}
\lim_{r\rightarrow\infty}\left(\frac{\partial W_p}{\partial r}-ik_p W_p\right)=0,\quad \lim_{r\rightarrow\infty}\left(\frac{\partial W_s}{\partial r}-ik_s W_s\right)=0,\quad r=|x|,
\en which holds uniformly in all directions $\hat{x}=x/|x|\in \mathbb{S}^{N-1}$. The functions $W_p$ and $W_s$ denote  the compressional and shear parts of $W$, respectively; see (\ref{kp}) and (\ref{ks}). It is well known that the above boundary value problem is well-posed for any $\varphi\in H^{-1/2}(\partial D_h)^N$.
\begin{lem}\label{Au-Le-2}
Let $W\in H^1_{loc}(\R^N\backslash\overline{D}_h)^N$ be the unique solution of the system (\ref{Lame-cavity})-(\ref{Rads}). Then, there exists $h_0\in\mathbb{R}_+$ such that when $h<h_0$,
\be\label{eq:35}
||W||_{H^{1/2}(\partial \Omega)^N}+||T W|_{\partial \Omega}||_{C(\partial \Omega)}\leq C\,h^{N-1}\,||\varphi(h\,\cdot\,)||_{H^{-3/2}(\partial D)^N},
\en where $C$ is a positive constant independent of $h$ and $\varphi$.
\end{lem}
\begin{proof}
From Betti's formula, we have the expression
\be\label{eq:34}
W(x)=(DL_{\partial D_h}\phi)(x)-(SL_{\partial D_h}\varphi)(x),\quad x\in \R^N\backslash\overline{D}_h,
\en with $\phi=W|_{\partial D_h}\in H^{1/2}(\partial D_h)^N$. Letting $x\rightarrow\partial D_h$ and applying the jump relations of layer potential operators, we obtain
\be\label{eq:integral1}
\frac{1}{2} \phi(x)=(K_{\partial D_h}\phi  )(x)-(S_{\partial D_h}\varphi  )(x),\quad x\in\partial D_h.
\en
Similar to (\ref{eq:23}), \eqref{eq:integral1} can be equivalently formulated as
\be\label{eq:29}
\left[\left(\frac{1}{2} I-K_{\partial D}^{(0)}-\mathcal{R}\right) \phi(h\,\cdot\,)\right](h x)
=-(S_{\partial D_h}\varphi)(hx),\quad x\in\partial D.
\en
where the kernel of the operator $\mathcal{R}: H^{-1/2}(\partial D)^N\rightarrow H^{-1/2}(\partial D)^N$ is given by the continuous matrix $\Pi^{(0)}-\Pi^{(\omega)}$, satisfying the estimate (see \cite[Chapter 4.3]{Nedelec})
\be\label{eq:37}
||\mathcal{R}||_{H^{-1/2}(\partial D)^N\rightarrow H^{-1/2}(\partial D)^N}\leq C\,h^{N-1}.
\en Since $\frac{1}{2}I-K_{\partial D}^{(0)}$ is an isomorphism from $H^{-1/2}(\partial D)^N$ to $H^{-1/2}(\partial D)^N$,
it follows from (\ref{eq:29}) and (\ref{eq:37}) that for $h\in\mathbb{R}_+$ sufficiently small
\be\label{eq:30}
||\phi(h\,\cdot\,)||_{H^{-1/2}(\partial D)^N}\leq C\, || (S_{\partial D_h}\varphi) (h\,\cdot\,) ||_{H^{-1/2}(\partial D)^N}.
\en
In order to estimate the left hand side of (\ref{eq:30}),
we decompose  $(S_{\partial D_h}\varphi) (hx)$ into
\be\no
(S_{\partial D_h}\varphi) (hx)&=&h^{N-1}(S_{\partial D}\varphi(h\,\cdot\,)) (hx)\\ \label{eq:31}
&=&h^{N-1}\,(S^{(0)}_{\partial D}\varphi(h\,\cdot\,)) (hx)+(\mathcal{G}_{\partial D}\varphi(h\,\cdot\,)) (hx),
\en where the integral kernel of the integral operator $\mathcal{G}_{\partial D}$ is given by
\ben
\Pi'(x,y)=h^{N-1}\Pi^{(\omega)}(hx,hy)-h^{N-1}\Pi^{(0)}(x,y).
\enn
Using (\ref{Gomega3}) and (\ref{Pi0}), together with
straightforward calculations, one can show that when $N=3$
\ben
\Pi'(x,y)&=&h^2\frac{1}{4\pi}\sum_{n\geq 2}^{\infty}\frac{(n+1)(\la+2\mu)+\mu}{\mu (\la+2\mu)}
\frac{(i\omega)^n}{(n+2)n!} |hx-hy|^{n-1}\, \textbf{I}\\
&&- h^2\frac{1}{4\pi}\sum_{n\geq 2}^{\infty}\frac{\la+\mu}{\mu (\la+2\mu)}\frac{(i\omega)^n (n-1)}{(n+2)n!} |hx-hy|^{n-3}\,(hx-hy)\otimes(hx-hy)\\
&=&h^2\,i\omega\frac{2\la+5\mu}{12\pi \mu (\la+2\mu)}\,\textbf{I}+h^2\omega^2\,A(h|x-y|),
\enn where $A$ is a real-analytic function satisfying $A(t)\rightarrow 0$ as $t\rightarrow 0$.
In two dimensions, it follows from (\ref{G0}) and (\ref{eq:14})  that
\ben
\Pi'(x,y)&=& h[\Pi^{(\omega)}(hx,hy)-\Pi^{(0)}(x,y) ]\\
&=& h\left[\Pi^{(0)}(hx,hy)-\Pi^{(0)}(x,y) +
\eta\,\textbf{I}\,+h^2\ln h\mathcal{O}(|x-y|^2\ln |x-y|)  \right]\\
&=& -\frac{3\mu+\la}{4\pi\mu(2\mu+\la)} h\ln h+ \eta\,\textbf{I} h+  h^3\ln h\mathcal{O}(|x-y|^2\ln |x-y|)\\
&=& \mathcal{O}(h\ln h)
\enn
as $h\rightarrow +0$.
Hence, by the mapping properties presented in \cite[Chapter 4.3]{Nedelec},
\be\label{eq:32}
||\left(\mathcal{G}_{\partial D}\varphi(hx)\right)(h\,\cdot\,) ||_{H^{-1/2}(\partial D)^N}\leq C\,e(h)\,||\varphi(h\,\cdot\,) ||_{H^{-3/2}(\partial D)^N},
\en with the dimensional constant
\ben
e(h):=\left\{ \begin{array}{lll}
h^2 &&\mbox{if}\quad N=3,\\
h\,\ln h &&\mbox{if}\quad N=2.
\end{array}\right.
\enn
Recalling the boundedness of the operator $S^{(0)}_{\partial D}: H^{-3/2}(\partial D)^N\rightarrow H^{-1/2}(\partial D)^N$, we see from (\ref{eq:31}) the estimate
\ben
||\left(S_{\partial D_h}\varphi\right)(h\,\cdot\,) ||_{H^{-1/2}(\partial D)^N}\leq
C\, \tilde{e}(h)\,||\varphi(h\,\cdot\,) ||_{H^{-3/2}(\partial D)^N}
\enn with
\ben
\tilde{e}(h):=\left\{ \begin{array}{lll}
h &&\mbox{if}\quad N=3,\\
h\,\ln h &&\mbox{if}\quad N=2.
\end{array}\right.
\enn
 Hence, by (\ref{eq:30}),
\be\label{eq:33}
||\phi(h\,\cdot\,)||_{H^{-1/2}(\partial D)^N}\leq C\, \tilde{e}(h)\,||\varphi(h\,\cdot\,) ||_{H^{-3/2}(\partial D)^N}.
\en
Let $\Omega_1$ be a compact set of $\R^N\backslash\overline{D}$ containing $\partial \Omega$. For $x\in \Omega_1$, we see from (\ref{eq:34}) and (\ref{eq:33}) that
\be\no
||W||_{L^2(\Omega_1)^N}&\leq& C\,h^{N-1}\{||\phi(h\,\cdot\,)||_{H^{-1/2}(\partial D)^N}+  ||\varphi(h\,\cdot\,) ||_{H^{-3/2}(\partial D)^N}\}\\ \label{eq:36}
&\leq&C\, h^{N-1}\,||\varphi(h\,\cdot\,) ||_{H^{-3/2}(\partial D)^N}.
\en
Finally, the estimate in (\ref{eq:35}) is a consequence of (\ref{eq:36}) and the interior estimate for elliptic boundary value problems.

The proof is complete.
\end{proof}

\subsection{Finite realization of the traction-free lining}
Finally, we present an interesting observation on the physical nature of the proposed lossy layer $\{D\backslash\overline{D}_{1/2}; \mathcal{C}^{(2)},$ $\rho^{(2)}\}=(F_h)_{*}\{ D_h\backslash\overline{D}_{h/2}; \tilde{\mathcal{C}}^{(2)}, \tilde{\rho}^{(2)} \}$ in our near-cloaking construction \eqref{GeneralOmega}-\eqref{construction}. It can be shown to be a finite realization of the traction-free lining. Indeed, we have

\begin{lem}
Suppose that $-\omega^2$ is not an eigenvalue of the elliptic operator $\mathcal{L}$ on $\Omega\backslash\overline{D}$ with the traction-free boundary condition. Let $U \in H^1(\Omega\backslash\overline{D})^3$ be the unique solution of
\be\label{Lame-3}\left\{\begin{array}{lll}
\nabla\cdot (\mathcal{C}^{(1)}: \nabla U )+\omega^2 \rho^{(1)}\, U =0&&\mbox{in}\quad \Omega\backslash\overline{D},\\ \mathcal{N}_{\mathcal{C}^{(1)}}\, U=\psi&&\mbox{on}\quad\partial \Omega,\\
\mathcal{N}_{\mathcal{C}^{(1)}}\, U=0&&\mbox{on}\quad\partial D,
\end{array}\right.
\en
where $\psi\in H^{-1/2}(\partial \Omega)^N$ and $\{\Omega\backslash\overline{D}; \mathcal{C}^{(1)},\rho^{(1)}\}$ is the elastic medium in (\ref{device-1}). Let $u \in H^1(\Omega)^N$ be the solution to the boundary value problem
\be\label{Lame-4}
\nabla\cdot (\mathcal{C}: \nabla u )+\omega^2 \rho\, u =0\quad\mbox{in}\quad \Omega,\quad
\mathcal{N}_{\mathcal{C}}\, u=\psi\quad\mbox{on}\quad\partial \Omega,
\en
where $(\Omega; \mathcal{C}, \rho)$ is given in (\ref{GeneralOmega})-\eqref{construction}.
Then for sufficiently small $h\in\mathbb{R}_+$, we have
\ben
||U -u ||_{H^{1/2}(\partial \Omega)^N}\leq C\, h^N ||\psi||_{H^{-1/2}(\partial \Omega)^N},
\enn
where $C$ is a positive constant independent of $h$ and $\psi$.
\end{lem}
\begin{proof} Set $\tilde{U}=F_* U, \tilde{u}=F_* u$ in $\Omega$. By Lemma \ref{Lem:change}, we see $\tilde{u}$ satisfies (\ref{Lame-2}) and $\tilde{U}$ is the solution of boundary value problem
\ben
\mathcal{L} \tilde{U} +\omega^2 \, \tilde{U}=0\quad\mbox{in}\quad \Omega\backslash\overline{D}_h,\quad
 T \tilde{U} =\psi\quad\mbox{on}\quad\partial \Omega,\quad
T \tilde{U}=0\quad\mbox{on}\quad\partial D_h.
\enn Moreover, we have $U=\tilde{U}$, $u=\tilde{u}$ on $\partial \Omega$.
Let $u_0$ be the solution to the free-space boundary value problem (\ref{Lame-0}). Then the difference $W:=u_0-\tilde{U}$ satisfies
\ben
\mathcal{L} W +\omega^2 \, W =0\quad\mbox{in}\quad \Omega\backslash\overline{D}_h,\quad
 T W =0\quad\mbox{on}\quad\partial \Omega,\quad
T W =T u_0\quad\mbox{on}\quad\partial D_h.
\enn From Lemma \ref{Au-Le-1} , we see
\be\label{eq:ef1}
||\tilde{U}-u_0||_{H^{1/2}(\partial \Omega)^N}=||W||_{H^{1/2}(\partial \Omega)^N}\leq C\,h^N\,||\psi||_{H^{-1/2}(\partial \Omega)^N}.
\en
On the other hand, it follows from Theorem \ref{Theorem-1} that
\be\label{eq:ef2}
||\tilde{u}-u_0||_{H^{1/2}(\partial \Omega)^N}\leq C\,h^N\,||\psi||_{H^{-1/2}(\partial \Omega)^N}.
\en
Hence, combining the two estimates in \eqref{eq:ef1} and \eqref{eq:ef2}, we finally have
\ben
&&||U-u||_{H^{1/2}(\partial \Omega)^N}\\
&=&||\tilde{U}-\tilde{u}||_{H^{1/2}(\partial \Omega)^N}\\
&\leq& ||\tilde{U}-u_0||_{H^{1/2}(\partial \Omega)^N}+||\tilde{u}-u_0||_{H^{1/2}(\partial \Omega)^N}\\
&\leq& C\,h^N\,||\psi||_{H^{-1/2}(\partial \Omega)^N}.
\enn

The proof is complete.
\end{proof}

\end{document}